\documentclass[10pt]{article}  
\usepackage{amsmath}
\usepackage{amssymb}
\usepackage{theorem}
\usepackage{euscript}
\usepackage{pstricks}

\newtheorem{theorem}{Theorem}[section]
\newtheorem{lemma}{Lemma}[section]
\newtheorem{corollary}{Corollary}[section]
\numberwithin{equation}{section}
\def\II{\mathbb I}

\def\ZZ{{\mathbb Z}}

\def\RRd{{\mathbb R}^d}

\def\NN{{\mathbb N}}

\def\II{{\mathbb I}}
\def\CC{{\mathbb C}}
\def\CC{{\mathbb C}}

\def\NN{{\mathbb N}}
\def\RR{{\mathbb R}}

\def\Vv{{\mathbb P}}
\def\UU{{\mathbb U}}
\def\Vv{{\mathcal P}}

\def\RRd{{\mathbb R}^d}

\def\IIi{{\mathbb I}^\infty}

\def\CCi{{\mathbb C}^\infty}

\def\UUi{{\mathbb U}^{\infty}}

\def\ZZip{{\mathbb Z}^\infty_+}

\def\Dd{{\mathcal D}}

\def\Ii{{\mathcal I}}

\def\Ll{{\mathcal L}}
\def\Oo{{\mathcal O}}
\def\Pp{{\mathcal P}}
\def\Ss{{\mathcal S}}

\def\Vv{{\mathcal V}}
\def\II{{\mathbb I}}
\def\CC{{\mathbb C}}
\def\ZZ{{\mathbb Z}}
\def\NN{{\mathbb N}}
\def\RR{{\mathbb R}}

\def\FF{{\mathbb F}}

\def\RRd{{\mathbb R}^d}

\def\dv{\operatorname{div}}
\def\LiD{L_\infty(D)}

\def\AAid{{\mathbb A}^\infty_\delta}
\def\AAidB{{\mathbb A}^\infty_{\delta,B}}

\def\Wi{W^1_\infty(D)}

\title{\sffamily Linear collective collocation and Galerkin approximations for \\
 parametric and stochastic elliptic PDEs} 
\author{ 
Dinh D\~ung \\[4mm]
Information Technology Institute, Vietnam National University, Hanoi \\
144 Xuan Thuy, Cau Giay, Hanoi, Vietnam\\
{\ttfamily dinhzung@gmail.com}\\[7mm] 
}

\date{\ttfamily  April 11, 2017 --  Version 4.2}
\begin{document}
\maketitle

\begin{abstract}
Consider the parametric elliptic problem 
\begin{equation} \nonumber
- \operatorname{div} \big(a(y)(x)\nabla u(y)(x)\big)
\ = \
f(x) \quad x \in D, \ y \in \IIi,
\quad u|_{\partial D} \ = \ 0, 
\end{equation}
where  $D \subset \RR^m$ is a bounded Lipschitz domain, $\IIi:=[-1,1]^\infty$, $f \in L_2(D)$, and the diffusions $a$ satisfy the uniform ellipticity assumption and are affinely dependent with respect to $y$. 
The parametric variable $y$ may be deterministic or random.
In the present paper, a central question to be studied is as follows.
Assume that we have an approximation property that there is a sequence of finite element approximations with a certain error convergence rate in energy norm of the space
$V:=H^1_0(D)$ for 
the nonparametric problem
$- \dv \big(a(y_0)(x)\nabla u(y_0)(x)\big)  = f(x)$
at every point $y_0 \in \IIi$.
Then under what assumptions does this sequence 
induce a sequence of finite element approximations with the same error convergence rate for 
the parametric elliptic problem  in the norm of the Bochner spaces $L_\infty(\IIi,V)$ or $L_2(\IIi,V)$? 
We solved this question by 
linear collective Taylor,  collocation and Galerkin methods, based on Taylor expansions, Lagrange polynomial interpolations and Legendre polynomial expansions, respectively, on the parametric domain $\IIi$. Under very light conditions, we show that all these approximation methods give the same error convergence rate  
as that by the sequence of finite element approximations for the nonparametric elliptic problem.  Hence the curse of dimensionality is broken by linear methods. 


  
\end{abstract}

\section{Introduction} 

In the recent decades, various approaches and methods have been proposed for the numerical solving of parametric partial differential equations 
of the form
\begin{equation} \label{General-SPDE}
\Dd(u,y)
\ = \
0, 
\end{equation}
where $u \mapsto \Dd(u,y)$ is a partial differential operator that depends on $d$ parameters represented as the vector
$y = (y_1,...,y_d) \in \Omega \subset \RRd$. If we assume that the problem \eqref{General-SPDE} is well-posed in a Banach space $X$, then the solution map $y \mapsto u(y)$ is defined from the parametric domain $\Omega$ to the solution space $X$. We refer the reader to \cite{CD15, GWZ14, SG11} for surveys and bibliography on different aspects in study of approximation and numerical methods for the problem \eqref{General-SPDE}.

Depending on the nature of the object modeled by the equation \eqref{General-SPDE},  the parameter $y$ may be either deterministic or random variable. The main challenge in numerical computation is to approximate the entire solution map $y \mapsto u(y)$ up to a prescribed accuracy with acceptable cost. This problem becomes actually difficult when $d$ may be very large. Here we suffer the so-called curse of dimensionality coined by Bellman: the computational cost grows exponentially in the dimension $d$ of the parametric space. Moreover, in some models the number of parameters may be even countably infinite.
In the present paper, a central question to be considered is: Under what assumptions does a sequence of finite element approximations  with a certain error convergence rate for 
the nonparametric problem
$\Dd(u,y_0) = 0$
at every point $y_0 \in \Omega$
induce a sequence of finite element approximations with the same error convergence rate for 
the parametric problem \eqref{General-SPDE}? 
We will solve it for a model parametric elliptic  equation by 
{\em linear collective methods}, and therefore, show that the curse of dimensionality is broken by them. However, we believe that our approach and methods can be extended to more general equations of the form \eqref{General-SPDE}. 

Let $D \subset \RR^m$ be a bounded domain with a Lipschitz boundary $\partial D$ and $\IIi:= [-1,1]^\infty$. Consider the parametric elliptic problem 
\begin{equation} \label{SPDE}
- \dv (a(y)\nabla u(y))
\ = \
f \quad \text{in} \quad D,
\quad u|_{\partial D} \ = \ 0, \quad y \in \IIi,
\end{equation}
where the gradient operator $\nabla$ is taken with respect to $x$, the diffusions $a(y)(x):=a(x,y)$  are functions of $x=(x_1,...,x_m) \in D$ and of parameters $y=(y_1,y_2,...) \in \IIi$ on $D \times \IIi$, and the function $f(x)$ is functions of $x=(x_1,...,x_m) \in D$. Throughout the present paper we preliminarily assume that $f \in L_2(D)$ and  the diffusions $a$ satisfy the {\em uniform ellipticity assumption}
\begin{equation} \label{UEA}
\quad 0 \ < \  r \ < \ a(y)(x)=a(x,y) \ \le \ R \ < \ \infty, \quad x \in D, \ y \in \IIi,
\end{equation}
and are affinely dependent with respect to $y$, or more precisely, 
\begin{equation} \label{KL-exp}
a(y)(x)
\ = \
\overline{a}(x) \ + \ 
\sum_{j=1}^\infty  y_j\, \psi_j(x), \quad x \in D, \ y \in \IIi, \quad \overline{a},\psi_j \in \Wi,
\end{equation}
where $\Wi$ is the space of functions $v$ on $D$, equipped with the semi-norm and norm 
\[ 
|v|_{\Wi}
:= \
\max_{1 \le i \le m} \|\partial_{x_i} v\|_{\LiD}, \quad 
\|v\|_{\Wi}
:= \
\|v\|_{\LiD} + |v|_{\Wi}. 
\] 

Based on finite element approximations with respect to the spatial variable $x$ and polynomial approximations with respect to the parametric variable $y$, there have been proposed several numerical methods for solving \eqref{SPDE}. Many works have been devoted to the development of the parametric Garlerkin and collocation techniques for the numerical solving of \eqref{SPDE}. As shown in \cite{DG15}, these methods are promising since they can use the possible regularity of the solution $u(y)$ with respect to the parameters $y$ to achieve faster convergence than sampling methods like Monte Carlo. A parametric Garlerkin method is a projection technique over a set of orthogonal polynomials with respect to an appropriate probability measure 
\cite{ABS09, BTZ04,  BNTT12, CCS13, CDS10, CDS11, DG15, FST05, Gi13, HoS12, MK05, NS13}. A collocation method is an approximation by a sum of Lagrangian interpolants based on the data of particular solution instances $u(y^{(i)})$ for some chosen values $y^{(1)},...,y^{(k)}$
\cite{BNT07, BNTT12, CCS13, EMPT11, MNST11, NTW08a, NTW08b}. In the case of problems with affine parameter dependence such as \eqref{KL-exp}, adaptive methods based on Taylor expansions have been investigated in 
\cite{CDS11, HS13a}.

In \cite{CCDS13}--\cite{CDS11}, \cite{HoS12} based on the $\ell_p$-assumption 
$\big(\|\psi_j\|_{\Wi}\big)_{j \in \NN} \in \ell_p(\NN)$ for some $0 < p <1$ on the affine expansion \eqref{KL-exp},
 the authors proposed nonlinear $n$-term approximation methods in energy norm by establishing {\em a priori} the set of the $n$ most useful infinite dimensional polynomials in Taylor expansion, Legendre polynomials expansion and Lagrange interpolation. The obtained $n$-term approximands then are approximated by finite element methods. 
It is worth to emphasize that the $\ell_p$-assumption crucially influences the convergence rate of the approximation error due to involving Stechkin's lemma. 
The results of \cite{CDS10, CDS11} have been improved \cite{BCM15, BCDS17} and extended to a class of parametric semi-linear elliptic PDEs \cite{HS13a, CCS15} and to parametric nonlinear PDEs \cite{CCS15,CD15}.
The reader can find a survey and bibliography on this direction in \cite{CD15}.

In the recent papers \cite{DG15,DGVR17}, we have considered a particular case of the equation \eqref{SPDE} where $D=[0,1]^m$, with an {\it a priori} assumption that the solution possesses higher order mixed smoothnesses of Sobolev-Korobov type or of Sobolev-analytic type simultaneously on spatial variable $x$ and parametric variable $y$. Applying results on hyperbolic cross approximation in infinite dimension, we constructed linear collective Galerkin methods on both variables $x$ and $y$ for approximation of the solution which give the convergences rate in energy norm as the same as that of approximation by Galerkin methods for solving the corresponding nonparametric elliptic problem the domain $[0,1]^m$. Moreover, the infinite-variate parametric part of the problem completely disappeared from the cost of complexities and influences only the constants.

Let $V:= H^1_0(D)$ and  denote by $W$ the subspace of $V$ equipped with the semi-norm and norm
\begin{equation} \nonumber
|v|_W 
:= \
\|\Delta v\|_{L_2(D)}, \quad 
\|v\|_W 
:= \
\|v\|_V + |v|_W.
\end{equation}
Assume that we have the following approximation property on the spatial domain $D$: There are a nested sequence of subspaces 
$(V_n)_{n \in \NN}$ in $V$,  a sequence of linear bounded operators 
$(P_n)_{n \in \NN}$ from $V$ into $V_n$, and a number 
 $0< \alpha \le 1/m$  such that  $\dim V_n \le n$  and
\begin{equation} \label{Int[W-approx-property]}
\|v - P_n(v)\|_V
\ \le \ 
C_D\, n^{-\alpha} \, \|v\|_W, \quad \forall v \in W.
\end{equation}
In the present paper, we propose collective Taylor, collocation and  Galerkin approximations in the Bochner spaces 
$L_\infty(\IIi,V)$ and $L_2(\IIi,V)$ for solving \eqref{SPDE}, based on this approximation property and Taylor expansions, Lagrange polynomial interpolations and  Legendre polynomials expansions, respectively, on the parametric domain $\IIi$. 
All the methods are linear and  constructive. The Taylor and Galerkin approximations are based on hyperbolic crosses, while the collocation method on sparse grids.  Moreover, they are collective with regard to spatial variable $x$ and parametric variable $y$. This means that in constructing these methods, the $m$-variate spatial part and the infinite-variate parametric  part are not separately but collectively treated. 

We put a light restriction on the diffusions $a(y)$: the inclusion 
\begin{equation} \label{ell_p-assumption}
\big(\|\psi_j\|_{\Wi}\big)_{j \in \NN} \in \ell_{p(\alpha)}(\NN)
\end{equation}
 with $p(\alpha) = \frac{1}{1+\alpha}$
for the collective Taylor and collocation approximations, and with 
$p(\alpha)= ,\frac{2}{1+2\alpha}$ for the collective Galerkin approximation.
 Under these conditions on the diffusions $a(y)$, we show that our methods give {\em the same convergence rate} $n^{-\alpha}$ of the error of the approximation of the solution of the nonparametric elliptic problem 
using the approximation property  \eqref{Int[W-approx-property]} 
(see \eqref{PDE} and \eqref{ApproxPDE} in Subsection~\ref{Nonparametric}). All the conditions on the diffusions $a(y)$ in particular, the $\ell_{p(\alpha)}$-assumption do not affect the convergence rate of the approximation error, completely disappear from it and influence only the constant. 
 Finally, notice also that the construction of linear collective approximations in the present paper is completely different from the construction of finite element approximations in \cite{CCS13}, \cite{CDS10}, \cite{CDS11}, and from the construction of linear collective approximations in \cite{DG15, DGVR17}.

The outline of the present paper is the following. In Section \ref{general approximation}, as a preliminary we investigate a general collective approximation in the space $L_\infty(\IIi,V)$. Section \ref{Taylor approximation} is devoted to the construction and error estimation of collective Taylor methods for solving 
\eqref{SPDE}. Section \ref{Collocation methods} is devoted to the construction and error estimation of collective collocation methods for solving \eqref{SPDE}.
In Section \ref{Legendre p.a.}, we extend the construction and methods in Section 
\ref{Taylor approximation}  to the construction and error estimation of  collective Legendre and Galerkin methods for solving \eqref{SPDE}. Section \ref{Concluding remarks} is devoted to some concluding remarks.

\section{A general collective approximation}
\label{general approximation}

\subsection{Nonparametric elliptic problem}
\label{Nonparametric}

Let us preliminarily consider the nonparametric complex-valued situation when we have only one equation:  
\begin{equation} \label{PDE}
- \dv (a\nabla u)
\ = \
f \quad \text{in} \quad D,
\quad u|_{\partial D} \ = \ 0, 
\end{equation}
where $f,a$ are complex-valued functions on $D$,  $f \in L_2(D)$  and  $a$ satisfies the ellipticity assumption
\begin{equation} \nonumber
 0 \ < \  r \ < \ \Re[a(x)] \ \le \ \ |a(x)| \ \le \
R \ < \ \infty, \quad x \in D.
\end{equation}
By the well-known Lax-Milgram lemma, there exists a unique solution $u \in V$ in weak form which satisfies the variational equation
\begin{equation} \nonumber
\int_{D} a(x)\nabla u(x) \cdot \nabla v(x) \, \mbox{d}x
\ = \
\int_{D} f(x) \, v(x) \, \mbox{d}x, \quad \forall v \in V.
\end{equation}
 Moreover, this solution satisfies the inequality
\begin{equation}  \label{|u(y)|_V <} 
\|u\|_V 
\ \le \
\frac{\|f\|_{V^*}}{r},
\end{equation}
where $V^* = H^{-1}(D)$ denotes the dual of $V$. 
Observe that there holds the embedding $L_2(D) \hookrightarrow V^*$ and the  inequality
$\|f\|_{V^*}  \le \|f\|_{L_2(D)}$.

If we assume that $a \in \Wi$, then the solution $u$ of \eqref{PDE} is in $W$. Moreover, $u$ satisfies the estimates
\begin{equation} \nonumber
|u|_W 
\ \le \
\frac{1}{r}\left(1 + \frac{|a|_{\Wi}}{r}\right) \|f\|_{L_2(D)},
\end{equation}
and
\begin{equation} \label{PDE[u_W<]}
\|u\|_W 
\ \le \
\frac{1}{r}\left[1 + \left(1 + \frac{|a|_{\Wi}}{r}\right)\right] \|f\|_{L_2(D)}.
\end{equation}


Suppose that we have an approximation property in the following assumption.

\smallskip
\noindent
{\bf Assumption (i)}: There are a nested sequence of subspaces 
$(V_n)_{n \in \NN}$ in $V$,  a sequence of linear bounded operators 
$(P_n)_{n \in \NN}$ from $V$ into $V_n$, and a number 
 $0< \alpha \le 1/m$  such that  $\dim V_n \le n$  and
\begin{equation}  \label{convergence-rate}
\|v - P_n(v)\|_V
\ \le \ 
C_D\, n^{-\alpha} \, \|v\|_W, \quad \forall v \in W,
\end{equation}
where $C_D$ is a constant which may depend on the domain $D$.

For example, classical error estimates \cite{Cia78} yield that
the convergence rate in \eqref{convergence-rate} with $\alpha=1/m$ can be achieved by using Lagrange finite elements on quasi-uniform partitions. 
Throughout the remainder of the present paper, $\alpha$ is fixed and used only for denoting the convergence rate in Assumption~(i). 

Under Assumption~(i) by C\'ea's lemma we have
\begin{equation} \label{ApproxPDE}
\|u - u_n\|_V
\ \le \ 
\sqrt{\frac{R}{r}} \, \inf_{v \in V_n}\|u - v\|_V
\ \le \ 
\sqrt{\frac{R}{r}}\, \|u - P_n(u)\|_V
\ \le \ 
\sqrt{\frac{R}{r}}\, C_D \, n^{-\alpha},
\end{equation}
where $u_n$ is the Galerkin approximation which is the unique solution of the problem
\begin{equation} \nonumber
\int_{D} a(x)\nabla u_n(x) \cdot \nabla v(x) \, \mbox{d}x
\ = \
\int_{D} f(x) \, v(x) \, \mbox{d}x, \quad \forall v \in V_n.
\end{equation}



\subsection{A collective approximation}
\label{A collective approximation}

We construct now a general collective linear method based on the approximation property  on the spatial domain $D$ in Assumption (i) and a unconditional expansion  on  the parametric domain $\IIi$.
To this end, for $k \in \ZZ_+$,  we define
\begin{equation} \nonumber
\delta_k (v)
:= \
P_{2^k} (v)  - P_{2^{k-1}} (v), \ k \in \NN, \quad \delta_0 (v) = P_0 (v).
\end{equation}
If Assumption (i) holds, then we can represent every $v \in W$  by the series
\begin{equation} \nonumber
v
\ = \
\sum_{k =0}^\infty \delta_k (v)
\end{equation}
converging in $V$ and satisfying the estimate
\begin{equation} \label{delta-approx-property}
\|\delta_k (v)\|_V
\ \le \
(1 + 2^\alpha) C_D\, 2^{-\alpha k} \, \|v\|_W, \quad k \in \ZZ_+.
\end{equation}

Denote by $\FF$ the subset in $\ZZip$ of all $s$ such that $\operatorname{supp}(s)$ is finite, where 
$\operatorname{supp}(s)$ is the support of $s$, that is the set of all $j \in \NN$ such that 
$s_j \not=0$.
We say that a sequence $(\Lambda_N)_{N \in \NN} \subset \FF$ of finite sets exhausts $\FF$ if any finite set 
$\Lambda \subset \FF$ is contained in all  $\Lambda_N$ for $N \ge N_0$ with $N_0$ sufficiently large. 
Similarly, we say that a sequence $(G_N)_{N \in \NN} \subset \ZZ_+ \times \FF$ of finite sets exhausts $\ZZ_+ \times \FF$ if any finite set $G \subset \ZZ_+ \times \FF$ is contained in all  $G_N$ for $N \ge N_0$ with $N_0$ sufficiently large.

For the normed space $X$ of functions on $D$, denote by $L_\infty(\IIi,X)$  the space of all mappings $v$ from 
$\IIi$ to $X$ for which the following norm is finite
\begin{equation} \nonumber
\|v\|_{L_\infty(\IIi,X)}
:= \
\sup_{y \in \IIi} \|v(y)\|_X.
\end{equation}
We also use the notation  
\begin{equation} \nonumber
|v|_{L_\infty(\IIi,X)}
:= \
\sup_{y \in \IIi} |v(y)|_X
\end{equation}
for a semi-norm $|v(y)|_X$ in $X$ if any.


\begin{lemma} \label{lemma[Uncond-convergence]}
Let Assumption $\operatorname{(i)}$ hold. 
Let $v \in L_\infty(\IIi,V)$ be represented as the series 
\begin{equation} \label{eq[Uncond-convergence](1)}
v(y)(x)
\ = \
\sum_{s \in \FF} g_s(x) \varphi_s(y)
\end{equation}
converging unconditionally in $L_\infty(\IIi,V)$
where
$g_s \in W$ and $(\|g_s\|_W )_{s \in \FF}$ belongs to $\ell_1(\FF)$, and $\varphi_s \in L_\infty(\IIi)$ with 
$\|\varphi_s\|_{L_\infty(\IIi)} = 1$. 
Then $v(y)$ can be represented as the series 
\begin{equation} \label{[g-series]}
v(y)(x)
\ = \
\sum_{(k,s) \in \ZZ_+ \times \FF} \delta_k (g_s)(x) \, \varphi_s(y), \quad y \in \IIi,
\end{equation}
converging unconditionally in $L_\infty(\IIi,V)$. 
\end{lemma}

\begin{proof}
Let us first prove the convergence of the series \eqref{[g-series]} for a sequence of special form 
$(G^*_N)_{N \in \NN}$ with
\begin{equation} \nonumber
G^*_N = ((k,s) \in \ZZ_+ \times \FF: \, 0 \le k \le N, \ s \in \Lambda_N),
\end{equation}
where $(\Lambda_N)_{N \in \NN}$ is any sequence of finite subsets in $\FF$ which exhausts 
$\FF$. 

We have for every $y \in \IIi$,
\begin{equation} \nonumber
\Big\|v(y) -  \sum_{(k,s) \in G^*_N}\delta_k (g_s) \varphi_s(y)\Big\|_V
\ \le \ 
\Big\|v(y) - \sum_{s \in \Lambda_N} g_s \varphi_s(y)\Big\|_V
\ + \ \Big\|\sum_{s \in \Lambda_N} g_s \varphi_s(y) - \sum_{(k,s) \in G^*_N}\delta_k (g_s) \varphi_s(y)\Big\|_V.
\end{equation}
Hence, due to the unconditional convergence \eqref{eq[Uncond-convergence](1)} it is sufficient to show that
\begin{equation} \label{lim-N-to-infty}
\lim_{N \to \infty}
\Big\|\sum_{s \in \Lambda_N} g_s \varphi_s - 
\sum_{(k,s) \in G^*_N}\delta_k (g_s) \varphi_s\Big\|_{L_\infty(\IIi,V)}
\ = \ 
0.
\end{equation}
 Using the assumptions of the lemma gives for every $y \in \IIi$,
\begin{equation} \nonumber
\begin{split}
\Big\|\sum_{s \in \Lambda_N} g_s \varphi_s(y) - \sum_{(k,s) \in G^*_N}\delta_k (g_s) \varphi_s(y)\Big\|_V
\ &= \
\Big\|\sum_{s \in \Lambda_N} g_s \varphi_s(y) - 
\sum_{s \in \Lambda_N} \sum_{k=0}^N \delta_k (g_s) \varphi_s(y) \Big\|_V 
\\[1.5ex]
\ &= \
\Big\|\sum_{s \in \Lambda_N} \Big[g_s - P_{2^N}(g_s)\Big] \, \varphi_s(y) \Big\|_V 
\ \le \
\sum_{s \in \Lambda_N} \|g_s - P_{2^N}(g_s)\|_V  
\\[1.5ex]
\ &\le \
\sum_{s \in \Lambda_N} C_D\, 2^{-\alpha N}\|g_s\|_W  
\ \le \
C_D\, \,2^{-\alpha N} \|(\|g_s\|_W)\|_{\ell_1(\FF)} 
\end{split}
\end{equation}
which proves  \eqref{lim-N-to-infty}. 

Let $(G_N)_{N \in \NN}$ be any sequence of finite subsets in $\ZZ_+ \times \FF$ which exhausts 
$\ZZ_+ \times\FF$. For any $\varepsilon > 0$, there exists $M = M(\varepsilon)$ such that
\begin{equation}\nonumber
\Big\|v -  \sum_{(k,s) \in G^*_M}\delta_k (g_s) \varphi_s\Big\|_{L_\infty(\IIi,V)}
\ \le \ 
\frac{\varepsilon}{2}.
\end{equation}
We have by \eqref{delta-approx-property} that
\begin{equation}\nonumber
\sum_{(k,s) \not\in G^*_M} \|\delta_k (g_s)\|_V 
 \ \le \
(2^\alpha + 1)C_D \, 2^{-\alpha M} \sum_{(k,s) \not\in G^*_M} \|g_s\|_W 
\ \le \ 
(2^\alpha + 1)C_D \, 2^{-\alpha M} \|(\|g_s\|_W)\|_{\ell_1(\FF)}. 
\end{equation}
Consequently, we may also assume  that
\begin{equation}\nonumber
\sum_{(k,s) \not\in G^*_M} \|\delta_k (g_s)\|_V 
 \ \le \
\frac{\varepsilon}{2}.
\end{equation}
Since $(G_N)_{N \in \NN}$ exhausts $\ZZ_+ \times\FF$, there exists $N^*$ such that $G^*_M \subset G_N$ for all
$N \ge N^*$.
Hence we derive that
\begin{equation}\nonumber
\Big\|v -  \sum_{(k,s) \in G_N}\delta_k (g_s) \varphi_s\Big\|_{L_\infty(\IIi,V)}
\ \le \ 
\Big\|v -  \sum_{(k,s) \in G^*_M}\delta_k (g_s) \varphi_s\Big\|_{L_\infty(\IIi,V)}
 + \sum_{(k,s) \not\in G^*_M} \|\delta_k (g_s)\|_V 
\ \le \ 
\varepsilon.
\end{equation}
The proof is complete.
\hfill 
\end{proof}

Let $v \in L_\infty(\IIi,V)$ be represented as the series \eqref{eq[Uncond-convergence](1)}
converging unconditionally in $L_\infty(\IIi,V)$
where
$g_s \in W$ and $(\|g_s\|_W )_{s \in \FF}$ belongs to $\ell_1(\FF)$, and $\varphi_s \in L_\infty(\IIi)$ with 
$\|\varphi_s\|_{L_\infty(\IIi)} = 1$. We know that if in addition Assumption $\operatorname{(i)}$ holds, $v$ can be represented by the series  \eqref{[g-series]}. We are interested in approximation of $v$ by its partial sums. To this end, for a finite subset $G$ in $ \ZZ_+ \times \FF$, we define the function 
\[
\Ss_G v(y)(x)
:= \
\sum_{(k,s) \in G} \delta_k (g_s)(x) \, \varphi_s(y).
\]
Notice that the function 
$\Ss_G v(y)(x)$  is defined collectively with regards to the spatial variables $x$ and parametric  variables $y$, i.e., $x$ and $y$ are not separated in constructing it. 

Let  $0 < p < \infty$ and $\sigma := (\sigma_s)_{s \in \FF}$ be a positive sequence. 
For $T>0$, define the following subset in $\ZZ_+ \times\FF$
\begin{equation} \label{def[G(T)]}
G(T)
\ = \ 
G_{p, \sigma}(T) 
:= \ 
\big\{(k,s) \in \ZZ_+ \times\FF: \, 2^k \sigma_s^p \leq T\big\}.
\end{equation}
Clearly, $G(T)$ is a finite set for every $T>0$. 
We will approximate $v$ by the partial sums $\Ss_{G(T)}v$ in the norm of $L_\infty(\IIi,V)$. To estimate the error of this approximation with regard to the parameter $T$, we will need the unconditional convergence of the series 
\eqref{[g-series]}  and a lemma on estimation of sums over the complement of the sets $G(T)$.

If $\alpha > 0$, $0 < p < 1$, we use the notation: $\alpha^*:= \alpha$ for $\alpha \le 1/p - 1$, 
and $\alpha^*:= \alpha - 1/p + 1$ for $\alpha > 1/p - 1$.

\begin{lemma} \label{lemma[sum-notin]}
Let  $\alpha > 0$, $0 < p < 1$ and $\sigma := (\sigma_s)_{s \in \FF}$ be a positive sequence
such that the sequence $\big(\sigma_s^{-1}\big)_{s \in \FF}$ belongs to $\ell_p(\FF)$. 
Then we have for every $T > 0$,
\begin{equation} \nonumber
\sum_{(k,s) \not\in G(T)} 2^{-\alpha k} \sigma_s^{-1}
\ \le \
C\, T^{-\min (1/p - 1, \alpha)},
\end{equation}
where
\begin{equation} \nonumber
C
:= \
\frac{1}{2^{\alpha^*} - 1} \, \big\|\big(\sigma_s^{-1}\big)\big\|_{\ell_p(\FF)}^p,
\end{equation}
\end{lemma}

\begin{proof}
We first consider the case $\alpha \le 1/p - 1$.
We have for every $N \in \NN$,
\begin{equation} \nonumber
\begin{split}
\sum_{(k,s) \not\in G(T)} 2^{-\alpha k} \sigma_s^{-1}
\ &\le \
\sum_{s \in \FF}   \sigma_s^{-1} \ 
\sum_{2^k \ > \ T\sigma_s^{-p}} 2^{-\alpha k} 
\ \le \
\sum_{s \in \FF} \ 
\frac{1}{2^\alpha - 1} \sigma_s^{-1} \Big(T\sigma_s^{-p}\Big)^{-\alpha} 
\\[2ex]
\ &= \
\frac{T^{-\alpha}}{2^\alpha - 1}\,
\sum_{s \in \FF} \sigma_s^{-(1-p\alpha)}
\ \le \ 
C\, T^{-\alpha}.   
\end{split}
\end{equation}
In the last step we used the inequality $1-p\alpha \ge p$. 

We next consider the case $\alpha > 1/p - 1$.
We have for every $N \in \NN$,
\begin{equation} \nonumber
\begin{split}
\sum_{(k,s) \not\in G(T)} 2^{-\alpha k} \sigma_s^{-1}
\ &\le \
\sum_{k \ \ge 0} 2^{-\alpha k} \ \sum_{\sigma_s \ge (T2^{-k})^{1/p}}   \sigma_s^{-1}  
\ \le \
\sum_{k \ \ge 0} 2^{-\alpha k} \ \sum_{\sigma_s \ge (T2^{-k})^{1/p}}   \sigma_s^{-(1-p)}  \sigma_s^{-p} 
\\[2ex]
\ &= \
T^{-(1/p-1)}\sum_{k \ \ge 0} 2^{-(\alpha - 1/p + 1) k} \  \sum_{\sigma_s \in \FF}     \sigma_s^{-p} 
\ \le \ 
C\, T^{-(1/p-1)}.   
\end{split}
\end{equation}
In the last step we used the inequality $\alpha - 1/p + 1 > 0$. 
\hfill
\end{proof}

The following theorem gives a upper bound of the approximation of $v$ by the approximant $\Ss_{G(T)} v$. 

\begin{theorem} \label{theorem[g-approximation]}
Let Assumption $\operatorname{(i)}$ hold. 
Let $v \in L_\infty(\IIi,V)$ be represented as the series 
\begin{equation} \nonumber
v(y)(x)
\ = \
\sum_{s \in \FF} g_s(x) \varphi_s(y)
\end{equation}
converging unconditionally in $L_\infty(\IIi,V)$
where $\varphi_s \in L_\infty(\IIi)$ with $\|\varphi_s\|_{L_\infty(\IIi)} = 1$. 
Let  the sequence 
 $\big(\sigma_s^{-1}\big)_{s \in \FF}$ belong to $\ell_p(\FF)$ for some $0 < p < 1$ and 
\begin{equation}  \label{ineq[g_s]}
\|g_s\|_W \le C' \sigma_s^{-1}, \ \forall s \in \FF.
\end{equation}
Then we have for every $T >0$,
\begin{equation}  \nonumber
\Big\|v - \Ss_{G(T)} v\Big\|_{L_\infty(\IIi,V)}
\le  
C\, T^{-\min (1/p - 1, \alpha)},
\end{equation}
where 
\begin{equation}  \nonumber
C:= \  C'\, C_D \,\frac{2^\alpha + 1}{2^{\alpha^*} - 1} \, \big\|(\sigma_s^{-1})\big\|_{\ell_p(\FF)}^p.
\end{equation}
\end{theorem}

\begin{proof}
By Lemma \ref{lemma[Uncond-convergence]} the series \eqref{[g-series]}
converging unconditionally in $L_\infty(\IIi,V)$ to $v$,
and therefore, we can write for every $y \in \IIi$,
\begin{equation} \nonumber
\Big\|v(y) - \Ss_{G(T)} v(y)\Big\|_V
\ = \
\Big\|\sum_{(k,s) \not\in G(T)} \delta_k (g_s) \, \varphi_s(y)\Big\|_V 
\ \le  \
\sum_{(k,s) \not\in G(T)} \|\delta_k (g_s)\|_V.  
\end{equation}
From \eqref{delta-approx-property} we derive that 
\begin{equation} \nonumber
\|\delta_k (g_s)\|_V 
\ \le \
(2^\alpha + 1) \, C_D\, \, 2^{-\alpha k} \|g_s\|_W,
\quad \forall (k,s) \in \ZZ_+ \times \FF.
\end{equation}
Hence, by \eqref{ineq[g_s]} we have that
\begin{equation} \nonumber
\Big\|v - \Ss_{G(T)} v\Big\|_{L_\infty(\IIi,V)}
\ \le \
(2^\alpha + 1) C_D\,C'\, \sum_{(k,s) \not\in G(T)} 2^{- \alpha k}\sigma_s^{-1}.
\end{equation}
By applying Lemma \ref{lemma[sum-notin]} we conclude the proof.
\hfill
\end{proof}

Our approximation strategy is as follows.
In the remainder of this paper, based on Assumption~(i) and Taylor expansion and  Lagrange polynomial interpolation  on the parametric domain $\IIi$, we will construct collective expansions of  the form \eqref{[g-series]} for linear collective Taylor and  collocation. 
In the next step, for each particular approximation, we will construct the linear approximation operators 
$\Ss_{G(T)} v(y)(x)$ where  $G(T) \ = \ G_{p, \sigma}(T)$ with  properly chosen  sequence 
$\sigma = (\sigma_s)_{s \in \FF}$ such that  there hold the assumptions 
of Theorem~\ref{theorem[g-approximation]}, in particular, the inequalities \eqref{ineq[g_s]} for the spatial components $g_s$. There may be many ways to construct  a sequence $\sigma = (\sigma_s)_{s \in \FF}$ satisfying
\eqref{ineq[g_s]}. Here, we suggest the way based on a direct estimate for $\|g_s\|_W$ derived from smoothness properties of the solution $u$. See \cite{Di15} for another way based on analytic regularities of solutions.

\section{Collective Taylor approximation}
\label{Taylor approximation}

We return now to the parametric equation \eqref{SPDE}. 
Let us extend the definition of the solution $u(y)$ to $u(z)$ for $z$ belonging to the unit polydics  
\begin{equation} \nonumber
\UUi
:= \
\{z=(z_1,z_2,...) \in \CCi: |z_j| \le 1, \ j \in \NN\}.
\end{equation}
To this end, based on the affine expansion \eqref{KL-exp} we extend $a(y)$ to $\CCi$ by  
\begin{equation} \label{AFF-exp}
a(z)(x)
\ = \
a(x,z)
:= \
\overline{a}(x) \ + \ 
\sum_{j=1}^\infty  z_j\, \psi_j(x), \quad x \in D, \ z \in \CCi, \quad \psi_j \in L_\infty(D).
\end{equation}
We then  consider the complex expansion of \eqref{SPDE}:
\begin{equation*} 
- \dv (a(z)\nabla u(z))
\ = \
f \quad \text{in} \quad D,
\quad u|_{\partial D} \ = \ 0, \quad z \in \UUi. 
\end{equation*}
Observe that the uniform ellipticity \eqref{UEA} 
implies the {\em complex uniform ellipticity}
\begin{equation} \label{SPDE-epllipticity}
 0 \ < \  r \ < \ \Re[a(x,z)] \ \le \ \ |a(x,z)| \ \le \
2R \ < \ \infty, \quad x \in D, \ z \in \UUi.
\end{equation}
Clearly, in the case where $f(x)$ and  $a(x,z)$ are real valued, the restriction of 
$u(z)$ to $\IIi$ coincides with $u(y)$. 



Since the complex uniform ellipticity 
\eqref{SPDE-epllipticity}
holds,  the map 
$z \mapsto u(z)$ is a $V$-valued and bounded analytic function in certain domains are larger than $\UUi$. 
Following \cite{CDS11}, for $0 < \delta \le r$, let us define 
\begin{equation} \nonumber
\AAid
:= \
\{z \in \CCi: \delta \ \le \ \Re[a(x,z)] \ \le \ \ |a(x,z)| \ \le \ 2R \}.
\end{equation}
Observe that $\UUi$ is contained in $\AAid$.

Due to the Lax-Migram lemma in complex form, if $f \in L_2(D)$ is given, then for all $z \in \AAid$ there exists a unique solution $u(z) \in V$ in weak form which satisfies the variational equation
\begin{equation} \label{var-eq[u(z)]}
\int_{D} a(x,z)\nabla u(x,z) \cdot \nabla v(x) \, \mbox{d}x
\ = \
\int_{D} f(x) \, v(x) \, \mbox{d}x, \quad \forall v \in V.
\end{equation}
Here and throughout we use the convention: $u(x,z):= u(z)(x)$. 
This solution also satisfies the inequality
\begin{equation} \nonumber
\|u(z)\|_V 
\ \le \
\frac{1}{\delta} \|f\|_{V^*}.
\end{equation}

For $0 < \delta < 2 \, R$ and $B > 0$, let us define 
\begin{equation} \nonumber
\AAidB
:= \
\{z \in \CCi: \, \delta  \le  \Re[a(x,z)]  \le  |a(x,z)|  \le  2 R, \quad 
|a(z)|_{\Wi} \le B, \quad \forall x \in D\}.
\end{equation}

We have seen in \eqref{PDE[u_W<]} that under the complex uniform ellipticity assumption \eqref{SPDE-epllipticity}, for $0 < \delta < r$ and sufficiently large $B$ the set $\AAidB$ is nonempty, $u(z) \in W$ for every $z \in \AAidB$, and moreover,
\begin{equation} \nonumber
\|u(z)\|_W 
\ \le \
C_{\delta,B}
:= \
\frac{1}{\delta}\left[1 + \left(1 + \frac{B}{\delta}\right)\right] \|f\|_{L_2(D)}, \quad z \in \AAidB.
\end{equation}

We consider the Taylor expansion of the solution $u(z)$ with respect to the parametric variable $z$.
For $s \in \FF$ with $\operatorname{supp}(s) \subset \{1,2,...,J\}$, we define the partial derivative
\[
\partial_z^s u
:= \
\frac{\partial^{|s|} u} {\partial^{s_1}{z_1} \cdots \partial^{s_J} {z_J}},
\]
where $|s| := \sum_{j =1}^J |s_j|$.
 We will need a condition for unconditional convergence towards $u(z)$ of the Taylor series
\begin{equation*} 
u(z)
\ = \
\sum_{s \in \FF} t_s \, z^s,
\end{equation*} 
where the Taylor coefficients $t_s$ are defined by
\[
t_s(x)
:= \
\frac{1}{s!} \partial_z^s u(0)(x)
\]
with $s ! := \prod_{j =1}^J  s_j !$ and $z^s:= \prod_{j =1}^J z_j^{s_j}$ 
(we use the convention $z_j^0:=1$ for  $z_j \in \CC$).

It was proven in \cite[Lemma 2.2]{CDS11} that at any $z \in \AAid$, the function $z \mapsto u(z)$ admits a complex derivative $\partial_{z_j} u(z) \in V$ with
respect to each variable $z_j$. This derivative is the weak solution of the problem: for $z \in \AAid$, find 
$\partial_{z_j} u(z) \in V$ such that
\begin{equation} \nonumber
\int_{D} a(x,z)\nabla \partial_{z_j} u(x,z) \cdot  \nabla v(x) \, \mbox{d}x
\ = \
- \int_{D} \psi_j(x) \, \nabla u(x,z) \cdot \nabla v(x) \, \mbox{d}x, \quad \forall v \in V.
\end{equation}
Hence, 
starting with $t_0:= u(0_\FF)$ 
we can recursively find all $t_s$ as the unique solution of the variational equation
\begin{equation} \nonumber
\int_{D} \bar{a}(x)\nabla t_s (x) \cdot \nabla v(x) \, \mbox{d}x
\ = \
- \sum_{j: \ s_j \not= 0}\int_{D} \psi_j(x) \, \nabla t_{s-e^j}(x) \cdot \nabla v(x) \, \mbox{d}x, 
\quad \forall v \in V,
\end{equation}
where $e^j \in \FF$ denotes the vector with value $1$ at position $j$ and $0$ otherwise.

\begin{lemma} \label{lemma[t-convergence]}
Assume that there exist a sequence $\sigma = (\sigma_s)_{s \in \FF}$ 
and a constant $C$ such that the sequence
$(\sigma_s^{-1})_{s \in \FF}$ belongs to $\ell_1(\FF)$ and
\begin{equation}  \nonumber
\|t_s\|_V
\ \le C \, \sigma_s^{-1}, \quad s \in \FF.
\end{equation}
Then $(\|t_s\|_V )_{s \in \FF}$ belongs to $\ell_1(\FF)$ and
\begin{equation} \label{u=Taylor-series}
u(y)(x)
\ = \
\sum_{s \in \FF} t_s(x)\, y^s, \quad x \in D, \quad y \in \IIi,
\end{equation}
converging unconditionally in $L_\infty(\IIi,V)$. 
\end{lemma}

\begin{proof}
This lemma can be proven in a similar way to the proof of  \cite[Theorem 1.3]{CDS11} which states that if 
$\big(\|\psi_j\|_{\LiD}\big)_{j \in \NN} \in \ell_p(\NN)$ for some $0 < p <1$, then $(\|t_s\|_V )_{s \in \FF}$ belongs to $\ell_p(\FF)$, 
and  there holds the expansion \eqref{u=Taylor-series} converging unconditionally in $L_\infty(\IIi,V)$. 
\hfill
\end{proof}


We derive the unconditional convergence of a collective expansion based on the approximation property on the spatial domain $D$ in Assumption (i) and the Taylor expansion on the parametric domain $\IIi$. 

\begin{lemma} \label{theorem[T-presentation]}
Let  Assumption $\operatorname{(i)}$ hold. Assume that there exist a sequence $\sigma = (\sigma_s)_{s \in \FF}$ 
and a constant $C$ such that the sequence
$(\sigma_s^{-1})_{s \in \FF}$ belongs to $\ell_1(\FF)$ and
\begin{equation}  \nonumber
\|t_s\|_W
\ \le C \, \sigma_s^{-1}, \quad s \in \FF.
\end{equation}
Then $(\|t_s\|_W )_{s \in \FF}$ belongs to $\ell_1(\FF)$ and
$u(y)$ can be represented as the series 
\begin{equation} \label{[T-series]}
u(y)(x)
\ = \
\sum_{(k,s) \in \ZZ_+ \times \FF} \delta_k (t_s)(x) \, y^s, \quad y \in \IIi,
\end{equation}
converging unconditionally in $L_\infty(\IIi,V)$.
\end{lemma}

\begin{proof}
This lemma follows from Lemmas \ref{lemma[t-convergence]} and \ref{lemma[Uncond-convergence]} by putting 
$v(y)(x) = u(y)(x)$, $g_s(x) = t_s(x)$ and $\varphi_s(y) = y^s$. 
\hfill 
\end{proof}

Based on the collective expansion \eqref{[T-series]} of the solution $u$, the approximation property 
\eqref{delta-approx-property} and an estimate for $\|t_s\|_W$ of the form
$\|t_s\|_W
 \le M \, \sigma_s^{-1}, \ s \in \FF,$
for which there holds the $\ell_p$-summability of the sequence $\sigma = (\sigma_s)_{s \in \FF}$ for some $0 < p < 1$,
we now construct linear collective Taylor approximations of the solution $u$ and estimate the approximation error in the norm of $L_\infty(\IIi,V)$, by applying the general theory
established in Subsection \ref{A collective approximation}. 
Let us formulate the exact condition on the sequences $\left(\|t_s\|_W\right)_{s \in \FF}$ as an assumption.

\smallskip
\noindent
{\bf Assumption (ii)}: There exist $0 < p < 1$, a sequence $\sigma = (\sigma_s)_{s \in \FF}$ 
and a constant $M$ such that the sequence
$(\sigma_s^{-1})_{s \in \FF}$ belongs to $\ell_p(\FF)$ and
\begin{equation}  \nonumber
\|t_s\|_W
\ \le M \, \sigma_s^{-1}, \quad s \in \FF.
\end{equation}

For a finite subset $G$ in $ \ZZ_+ \times \FF$, 
denote by $\Vv^{{\rm T}}(G)$ the subspace in $L_\infty(\IIi,V)$ of  all functions $v$
of the form
\begin{equation} \nonumber
v(y)(x)
\ = \
v(x,y)
\ = \
\sum_{(k,s) \in G} v_k(x) \, y^s, \quad y \in \IIi, \quad v_k \in V_{2^k},
\end{equation}
and define the linear operator $\Ss^{{\rm T}}_G: \, L_\infty(\IIi,V) \to \Vv^{{\rm T}}(G)$ by
\[
\Ss^{{\rm T}}_G u(y)(x)
\ = \
\Ss^{{\rm T}}_G u(x,y)
:= \
\sum_{(k,s) \in G} \delta_k (t_s)(x) \, y^s.
\]

\begin{theorem} \label{theorem[T-approximation](1)}
Let Assumptions $\operatorname{(i)}$ and $\operatorname{(ii)}$ hold. 
For $T > 0$, consider the set $G(T) = G_{p,\sigma}(T)$ as in \eqref{def[G(T)]}.   
Then we have for every $T > 0$,
\begin{equation}  \nonumber
\Big\|u(y) - \Ss^{{\rm T}}_{G(T)} u(y)\Big\|_{L_\infty(\IIi,V)}
\ \le \ 
C\, T^{-\min (1/p - 1, \alpha)},
\end{equation}
where
\begin{equation}  \nonumber
C:= \  M\, C_D \,
\frac{2^\alpha + 1}{2^{\alpha^*} - 1} \, \big\|\big(\sigma_s^{-1}\big)\big\|_{\ell_p(\FF)}^p.
\end{equation}
\end{theorem}

\begin{proof}
Again, put $v(y)(x) = u(y)(x)$, $g_s(x) = t_s(x)$ and $\varphi_s(y) = y^s$. By Lemma \ref{theorem[T-presentation]}
the series \eqref{[g-series]} converges unconditionally in $L_\infty(\IIi,V)$ to $v$.
By applying Theorem \ref{theorem[g-approximation]} we prove the theorem.
\hfill 
\end{proof}

We show that under the assumptions of  Theorem \ref{theorem[T-approximation](1)}, for a given $n \in \NN$, the respective operator $\Ss^{{\rm T}}_{G(T_n)}$ with properly chosen $T_n$ is a bounded linear operator in $L_\infty(\IIi,V)$ of rank $\le n$ which gives the convergence rate of the approximation to $u(y)$ as $n^{- \alpha}$. For any $n \in \NN$, let $T_n$ be the number defined by the inequalities 
\begin{equation} \label{[T_n]}
2 \big\|\big(\sigma_s^{-1}\big)\big\|_{\ell_p(\FF)}^p \, T_n
\ \le \
n
\ < \ 4 \big\|\big(\sigma_s^{-1}\big)\big\|_{\ell_p(\FF)}^p \, T_n.
\end{equation}

We need also an estimate of the rank of the linear operator $\Ss^{{\rm T}}_{G(T)}$ which is not larger that the sum
$\sum_{(k,s) \in G(T)} 2^k$, in the following lemma.

\begin{lemma} \label{lemma[sumG(T)<]}
Let $0 < p < \infty$, let $\sigma = (\sigma_s)_{s \in \FF}$  be a positive sequence such that the sequence
 $\big(\sigma_s^{-1}\big)_{s \in \FF}$ belongs to $\ell_p(\FF)$. 
Then we have for $G(T):= G_{p,\sigma}(T)$ and for every $T > 0$,
\begin{equation} \nonumber
\sum_{(k,s) \in G(T)} 2^k
\ \le \
C \, T,
\end{equation}
where 
\begin{equation} \nonumber
C:= 2\big\|\big(\sigma_s^{-1}\big)\big\|_{\ell_p(\FF)}^p.
\end{equation}
\end{lemma}

\begin{proof}
The lemma is trivial for $T < 1$ since in this case the set $G(T)$ is empty. Let us prove it for $T \ge 1$. 
Note that, for every $(k,s) \in G(T)$,  it follows from the definition of 
$G(T)$ that  $\sigma_s^p \ \le \ T$.
Hence, we have
\begin{equation} \nonumber
\sum_{(k,s) \in G(T)} 2^k
\ \le  \
\sum_{\sigma_s^p \ \le \ T} \quad 
\sum_{2^k \ \le T\sigma_s^{-p}} 2^k 
\  \le \
2 \sum_{\sigma_s^p \ \le \ T}  T\sigma_s^{-p} 
\ \le \
2\,T \sum_{s \in  \FF} \sigma_s^{-p}
\ \le \
C\,T.
\end{equation}
\hfill
\end{proof}

\begin{theorem} \label{theorem[Galerkin-approximation]}
Let  the assumptions and notations of  
Theorem \ref{theorem[T-approximation](1)} hold. 
For any $n \in \NN$, let $T_n$ be the number defined as in \eqref{[T_n]}
and put $\Vv^{{\rm T}}_n:= \Vv^{{\rm T}}\big(G(T_n)\big)$, \  $\Pp_n:= \Ss^{{\rm T}}_{G(T_n)}$. Then there holds the following.
\begin{itemize}
\item 
$\big\{\Vv^{{\rm T}}_n\big\}_{n \in \ZZ_+}$ is a nested sequence of subspaces in 
$L_\infty(\IIi,V)$ and $\dim \Vv^{{\rm T}}_n \le n$;

\item
$\big\{\Pp_n\big\}_{n \in \ZZ_+}$ is a sequence of linear bounded operators
from $L_\infty(\IIi,V)$ into $\Vv^{{\rm T}}_n$;  and

\item
For every $n \in \NN$, 
\begin{equation}  \nonumber
\big\|u - \Pp_n u\big\|_{L_\infty(\IIi,V)}
\ \le \ 
C\, n^{-\min (1/p - 1, \alpha)},
\end{equation}
with the same $\alpha$ as in the convergence rate  of the approximation 
in Assumption $\operatorname{(i)}$ and $p$ as in  Assumption $\operatorname{(ii)}$,
where 
\begin{equation}  \label{[constant-C]}
C:= \ M \, C_D \,
4^\alpha \ \frac{2^\alpha + 1}{2^{\alpha^*} - 1} \, \big\|\big(\sigma_s^{-1}\big)\big\|_{\ell_p(\FF)}^{p\alpha}.
\end{equation}
Moreover, if in addition,  $p = \frac{1}{1+\alpha}$ in  Assumption $\operatorname{(ii)}$, then we have that
\begin{equation}   \nonumber
\big\|u - \Pp_n u\big\|_{L_\infty(\IIi,V)}
\ \le \ 
C\, n^{- \alpha},
\end{equation}
\end{itemize}
\end{theorem}

\begin{proof}
From Assumption (i) we have
\begin{equation} \nonumber
\dim \Vv^{{\rm T}}(G(T_n))
\ \le \ \sum_{(k,s) \in G(T_n)} \dim V_{2^k} 
\ \le \ \sum_{(k,s) \in G(T_n)} 2^k.
\end{equation}
Hence, by Lemma \ref{lemma[sumG(T)<]} and \eqref{[T_n]} we derive that
\begin{equation} \label{[dim-P]}
\dim \Vv^{{\rm T}}(G(T_n))
\ \le \ 2  \big\|\big(\sigma_s^{-1}\big)\big\|_{\ell_p(\FF)}^p \, T_n
\ \le \ n.
\end{equation}
On the other hand, by \eqref{[T_n]},
\begin{equation} \nonumber
T_n^{-\min (1/p - 1, \alpha)} 
\ \le \ 4^\alpha  \big\|\big(\sigma_s^{-1}\big)\big\|_{\ell_p(\FF)}^{p\alpha} \, n^{-\min (1/p - 1, \alpha)}
\end{equation}
which together with Theorem  \ref{theorem[T-approximation](1)} 
and \eqref{[dim-P]} completes the proof of the theorem.
\hfill
\end{proof}

Observe that as in \eqref{[Pp_n-approximation](1)} the approximation methods $\Pp_n$ of $u(y)$ give the same convergence rate as that by the approximation methods $P_n$ in Assumption (i) for solving the corresponding nonparametric elliptic problem in the domain $D$. The parametric infinite-variate part as well as Assumption (ii) do not affect the convergence rate, completely disappears from it and influence only the constant $C$ given in  
\eqref{[constant-C]}. 

From Theorem \ref{theorem[T-approximation](1)} we see that under Assumption~(i) the problem of construction of a linear collective Taylor approximation is reduced to find a number $0 < p <1$, a constant $M$ and  a sequence
$\sigma = (\sigma_s)_{s \in \FF}$ satisfying Assumption~(ii).
We present a way based an estimate for $\|\partial_y^s u\|_{L_\infty(\IIi,W)}$ (see  \cite[Theorem 8.2]{CDS10} for a similar estimate). We define the following constant $K$ and sequence $b$ as follows.
\begin{equation} \label{def[K]}
K:= \
\frac{1}{r}\left[1 + \left(1 + \frac{|a|_{L_\infty(\IIi,\Wi)}}{r}\right)\right] \|f\|_{L_2(D)}; 
\end{equation}
\begin{equation} \label{def[b]}
b \ = \ (b_j)_{j \in \NN}, \quad
b_j 
:= \
\frac{1}{r} \left(\left(\frac{|a|_{L_\infty(\IIi,\Wi)}}{r} 
+ 2\right)\|\psi_j\|_{L_\infty(D)} + |\psi_j|_{\Wi}\right).
\end{equation}

\begin{lemma} \label{lemma|partial^s-u|_W}
Assume that $a \in L_\infty(\IIi,\Wi)$. Then we have
\begin{equation} \nonumber
\|\partial_y^s u\|_{L_\infty(\IIi,W)} 
\ \le \
K |s|!\, b^s, \quad s \in \FF. 
\end{equation}
\end{lemma}

\begin{proof}
Let us prove the lemma by induction on  $|s|$. For $s = 0_{\FF}$, from \eqref{PDE[u_W<]} we derive that
\begin{equation} \nonumber
\|u\|_{L_\infty(\IIi,W)} 
\ \le \
\frac{1}{r}\left[1 + \left(1 + \frac{|a|_{L_\infty(\IIi,\Wi)}}{r}\right)\right] \|f\|_{L_2(D)} 
\ = \ K.
\end{equation}
Suppose that the lemma holds true for all $\nu \in \FF$ with $|\nu| < |s|$. We will prove it for $s$.
Let a $k \in \ZZ^m_+$ with $|k| \le \nu - 2$ be given. Taking differentiation both sides of the equation
By \eqref{SPDE} we get
\begin{equation} \nonumber 
- a\Delta u
\ = \
f + \nabla a \cdot \nabla u,
\end{equation} 
we obtain 
\begin{equation} \nonumber 
- \partial_y^s \big(a\Delta u\big)
\ = \
 \partial_y^s\big(\nabla a \cdot \nabla u).
\end{equation}
Applying the Leibniz rule of multivariate differentiation 
to the both sides we obtain
\begin{equation} \nonumber 
- \sum_{0 \le \nu \le s} \binom{k}{\nu}\partial_y^\nu a \,\partial_y^{s-\nu} \big(\Delta u\big)
\ = \
 \sum_{0 \le \nu \le s} \binom{s}{\nu}
\partial_y^\nu (\nabla a) \cdot \partial_y^{k-\nu} (\nabla u).
\end{equation}
Hence, due to \eqref{AFF-exp} we get
\begin{equation} \nonumber 
- a \,\Delta\big(\partial_y^s u\big)
\ = \
\sum_{j: \ s_j \not=0} s_j \psi_j \Delta\big(\partial_y^{s-e^j} u\big)
\ + \
\nabla a \cdot \nabla\big(\partial_y^s u\big)
\ + \
\sum_{j: \ s_j \not=0} s_j \nabla \psi_j \cdot \nabla\big(\partial_y^{s-e^j} u\big)
\end{equation}
which implies that
\begin{equation} \nonumber
\begin{split} 
r \, |\partial_y^s u|_{L_\infty(\IIi,W)}
\ &\le \ 
|a|_{L_\infty(\IIi,\Wi)} \, \|\partial_y^s u\big\|_{L_\infty(\IIi,V)}
\\[1.5ex]
\ &+ \
\sum_{j: \ s_j \not=0} s_j \|\psi_j\|_{L_\infty(D)} |\partial_y^{s-e^j}u|_{L_\infty(\IIi,W)}
\ + \
\sum_{j: \ s_j \not=0} s_j |\psi_j|_{\Wi} \, \|\partial_y^{s-e^j} u\|_{L_\infty(\IIi,V)}
\end{split}
\end{equation}
It has been proven in  \cite[(4.11)]{CDS10} that
\begin{equation} \nonumber
\|\partial_y^s u\big\|_{L_\infty(\IIi,V)}
\ \le \
\sum_{j: \ s_j \not=0} s_j \frac{\|\psi_j\|_{L_\infty(D)}}{r} \|\partial_y^{s-e^j} u\|_{L_\infty(\IIi,V)}.
\end{equation}
All these together with the induction assumption give
\begin{equation} \nonumber
\begin{split} 
\|\partial_y^s u\|_{L_\infty(\IIi,W)}
\ & \le \ 
 \|\partial_y^s u\|_{L_\infty(\IIi,V)} + |\partial_y^s u|_{L_\infty(\IIi,W)}
 \ \le \
\sum_{j: \ s_j \not=0} s_j b_j \|\partial_y^{s-e^j}u\|_{L_\infty(\IIi,W)}
\\[1.5ex]
\ &\le \ 
K \sum_{j: \ s_j \not=0} s_j b_j (|s| - 1)! \, b^{s-e^j}
\ \le \ 
K \sum_{j: \ s_j \not=0} s_j  (|s| - 1)! \, b^s \ = \ K |s|!\, b^s.
\end{split}
\end{equation}
\hfill
\end{proof}

\begin{lemma} \label{ell_p-sumability}
Let  $0<p<\infty$, $c = (c_j)_{j \in \NN}$  be a positive sequence. Then we have the following.
\begin{equation}\nonumber
\left(\frac{|s|!}{s!}\, c^s\right) \in \ell_p(\FF) \ \Longleftrightarrow \
\begin{cases} 
\|c\|_{\ell_1(\NN)} \ < \ 1, \ c \in \ell_p(\NN), \ & \ \text{for}\  p \le 1; 
\\[1ex]
\|c\|_{\ell_1(\NN)} \ \le \ 1,  \ & \ \text{for} \ p > 1.
\end{cases}
\end{equation}
\end{lemma}

This lemma was proved in \cite[Theorem 7.2]{CDS10} for $p \le 1$ and in \cite[Theorem 5.2]{DGVR17} for $p>1$.
From it and Lemma~\ref{lemma|partial^s-u|_W} we obtain

\begin{corollary}  
Let the function 
$a$  belong to $L_\infty(\IIi,\Wi)$. Assume that there exists $0 < p <1$ such that the sequence  
$\big(\|\psi_j\|_{\Wi}\big)_{j \in \NN}$  belong to $\ell_p(\NN)$ and that $\|b\|_{\ell_1(\NN)} \ < \ 1$.
Then there holds Assumption~$\operatorname{(ii)}$ for $p$, $M=K$ and the sequence
\begin{equation}  \nonumber
\sigma:= \big(\sigma_s\big)_{s \in \FF}, \quad
\sigma_s^{-1}
:= \
\frac{|s|!}{s!} b^s.
\end{equation}
\end{corollary}

\section{Collective collocation approximation}
\label{Collocation methods}

\subsection{Tensorisation}

Our collective collocation method of polynomial  interpolation is based on the approximation property in Assumption (i) and on the standard principle of tensorisation of  difference of successive one-dimensional interpolation operators introduced in \cite{CCS13}. Let us recall it as well some auxiliary results from there. 

Let $(\xi_j)_{j \in \ZZ_+}$ be a sequence of mutually distinct points in $\II$. Then the univariate Lagrange interpolation operator $I_k$ associated with the section $\{\xi_0,...,\xi_k\}$, is defined by
\begin{equation} \nonumber
I_kv
:= \
\sum_{j=0}^k v(\xi_j)\, \ell_j^k, \quad
\ell_j^k(y)
:= \
\prod_{\substack{i=0 \\ i \not= j}}^k \frac{y - \xi_i}{\xi_j - \xi_i}
\end{equation}
for a function $v$ defined on $\II$. For $s \in \ZZ_+$, let us introduce the difference operator
\begin{equation} \nonumber
\Delta_s
:= \
I_s - I_{s-1}
\end{equation}
with the convention $I_{-1}(v) = 0$. 
If $s \in \FF$, we define 
\begin{equation} \nonumber
\xi_s
:= \
(\xi_{s_j})_{j \in \ZZ_+} \in \IIi,
\end{equation}
and the tensor product difference operator
\begin{equation} \nonumber
\Delta_s
:= \
\bigotimes_{j \in \NN} \Delta_{s_j}.
\end{equation}
For a finite set $\Lambda \subset \FF$, we introduce the interpolation operator
\begin{equation} \nonumber
I_\Lambda
:= \
\sum_{s \in \Lambda} \Delta_s,
\end{equation}
the space of polynomials
\begin{equation} \nonumber
V_\Lambda
:= \
\operatorname{span} \{y^s: \, s \in \Lambda\}
\end{equation}
and the grid
\begin{equation} \nonumber
\Gamma_\Lambda
:= \
\{\xi_s: \, s \in \Lambda\}.
\end{equation}

A set $\Lambda \subset \FF$  is called lower if  $s \in \Lambda$, then $s-e^j \in \Lambda$ for every $j$ such that $s_j > 0$. 
 For every lower set $\Lambda$, the generalization of the interpolation operator $I_\Lambda$ to the $V$-valued setting is straightforward: $I_\Lambda u$ is the unique solution in $V_\Lambda$ that coincides with $u$ at the points $\xi_s$ for $s \in \Lambda$. 

Denote by $L_\infty(\IIi)$ the space of bounded complex-valued functions $v$ on $\IIi$ equipped with the sup norm 
\[
\|v\|_{L_\infty(\IIi)}:= \sup_{y \in \IIi} |v(y)|.
\]
Then the Lebesgue constant of the interpolation operator 
$I_\Lambda$ is defined as
\begin{equation} \nonumber
\Ll_\Lambda
:= \
\sup_{\|v\|_{L_\infty(\IIi)} \le 1} \|I_\Lambda v\|_{L_\infty(\IIi)}.
\end{equation}
We are interested in selecting sequences  $(\xi_j)_{j \in \ZZ_+}$ so that the univariate Lebesgue constants
\begin{equation} \nonumber
\lambda_k
:= \
\sup_{\|v\|_{L_\infty(\II)} \le 1} \|I_k v\|_{L_\infty(\II)},
\end{equation}
associated with the univariate operator $I_k$ are increasing moderately. (Note that $\lambda_0 = 0$ for any choice of  $(\xi_j)_{j \in \ZZ_+}$). Such a sequence is the projection of a Leja sequence on the complex disk $\UU$ to $\II$ which is defined inductively by fixing a point $z_0 \in \UU$ and defining
\begin{equation} \nonumber
z_k
:= \
\operatorname{Argmax}_{z \in \UU}\, \prod_{j=0}^{k-1} |z - z_j|.
\end{equation}
The following lemma has been proven in \cite{CCS13}.
\begin{lemma} \label{lemma[L-constant]}
Let the Lebesgue constants $\lambda_k$ satisfy
$\lambda_k \le (k + 1)^\theta$, $k \in \ZZ_+$, for some $\theta \ge 1$. Then 
$\Ll_\Lambda \le |\Lambda|^{\theta + 1}$ for every lower set $\Lambda$.
\end{lemma}

Upper bounds of the form $\lambda_k \le (k + 1)^\theta$ can be derived for some $\theta > 0$ from the fact that 
$\lambda_k = \Oo(k^\gamma)$ for some $\gamma >0$. For the sequence $(\xi_j)_{j \in \ZZ_+}$ given by the projection of the  Leja sequence on the complex disk $\UU$ to $\II$ with $z_0=1$, it has been proven in \cite{Chk13} that
\[
\lambda_k
\ \le \
3(k+1)^2\, \log(k+1)
\ = \ 
\Oo(k^{2 + \varepsilon})
\]
for arbitrary fixed $\varepsilon > 0$.

For $k \in \ZZ_+$, we introduce the univariate polynomials $h_k$ of degree $k$ associated with the sequence 
$(\xi_j)_{j \in \ZZ_+}$ by
\begin{equation} \nonumber
h_0(y)
:= \ 1, \quad 
h_k(y)
:= \
\prod_{j=0}^{k-1} \frac{y - \xi_j}{\xi_k - \xi_j}, \ k \in \NN.
\end{equation}
If $s \in \FF$, we define 
the tensor product function
\begin{equation} \nonumber
h_s(y)
:= \
\prod_{j \in \NN} h_{s_j}(y_j).
\end{equation}

Let $(\Lambda_n)_{n \in \NN}$ be a nested sequence of sets with $n = |\Lambda_n|$. 
Then the grids $(\Gamma_{\Lambda_n})_{n \in \NN}$ are also nested. Note that each set $\Lambda_n$ can be seen as the section $\{s^1,...,s^n\}$ of a sequence  $(s^k)_{k \in \NN}$. This allows us to construct an algorithm for the computation $I_{\Lambda_n}v$ from $I_{\Lambda_{n-1}}v$. Namely, the polynomials $I_{\Lambda_n}v$ can be given by 
\begin{equation} \nonumber
I_{\Lambda_n}v
\ = \ 
\sum_{s \in \Lambda_n} v_s h_s
\ = \
\sum_{k=1}^n v_{s^k} h_{s^k},
\end{equation}
where $v_{s^k}$ are defined recursively by
\begin{equation} \nonumber
\begin{split}
v_{s^1} 
&:= \
v(\xi_0)
\\ 
v_{s^{k+1}}
&:= \
v(\xi_{s^{k+1}}) - I_{\Lambda_k}v(\xi_{s^{k+1}})
\ = \
v(\xi_{s^{k+1}}) - \sum_{j=1}^k v_{s^j}\, h_{s^j}(\xi_{s^{k+1}}).
\end{split}
\end{equation}

\subsection{An estimate of Taylor coefficients}

Following \cite{CDS11}, for $\rho:= (\rho_j)_{j \in \NN}$ be a sequence of positive numbers and $\delta, B > 0$, we say that $\rho$ is 
$(\delta,B)$-admissible, if 
\begin{equation} \label{dB-admit1}
\sum_{j=1}^\infty  \rho_j\, |\psi_j(x)|
 \ \le \ 
\Re[\overline{a}(x)] 
\ - \ \delta,  \quad x \in D,
\end{equation}
and 
\begin{equation}  \label{dB-admit2}
\sum_{j=1}^\infty  \rho_j\, \max_{1 \le i \le m}|\partial_{x_i}\psi_j(x)|
 \ \le \ 
B \ - \ | \overline{a}|_{\Wi},  \quad x \in D.
\end{equation}

For a proof of the following lemma, see  \cite[Lemma 5.4]{CDS11}. 
\begin{lemma} \label{lemma[t_s]-W}
Let  $\rho:= (\rho_j)_{j \in \NN}$ be 
$(\delta,B)$-admissible for $0 < \delta < r$ and sufficiently large $B$. Then we have
\begin{equation} \nonumber
\|t_s\|_W 
\ \le \
C_{\delta,B} \, \rho^{-s}, \quad s \in \FF. 
\end{equation}
\end{lemma}

 Assume that the sequence 
$\big(\|\psi_j\|_{\Wi}\big)_{j \in \NN}$ belongs to $\ell_1(\NN)$. For a given number $q > 1$, let us give an estimate for $p_s\|t_s\|_W$, $s \in \FF$, where
\[
p_s
:= \
\prod_{j \in \NN} (s_j + 1)^q.
\]
By the assumptions we may choose $\lambda > 1$ and $j_0$  so that
\begin{equation}  \nonumber
(\lambda - 1)\sum_{j \in E} \|\psi_j\|_{\Wi}
\ \le \ 
\frac{r}{6}, \quad \sum_{j > j_0} \|\psi_j\|_{\Wi}
\ \le \ 
\frac{r}{12e^q}.
\end{equation}
We split $\NN$ into the two sets $E:=\{1,...,j_0\}$ and $F:= \{j_0,j_0 + 1,...\}$, and for each $s \in \FF$ define the sequence $\rho = \rho(s)$ by
\begin{equation}  \nonumber
\rho_j
:= \
\begin{cases}
\lambda, \ & j \in E, \\[0.5ex]
1, \ & j \in F, \ s_j = 0, \\[0.5ex]
e^q + \frac{rs_j}{4|s_F|\|\psi_j\|_{\Wi}},\ & j \in F, \ s_j \not= 0.
\end{cases}
\end{equation}
Let us show that $\rho$ is $(r/2,B)$-admissible, where 
\begin{equation} \nonumber
B
:= \
\|\bar{a}\|_{\Wi} + \sum_{j \in E} \|\psi_j\|_{\Wi} + \frac{r}{2}.
\end{equation}
We verify, for instance the condition \eqref{dB-admit2}, the condition \eqref{dB-admit1} can be verified in a similar way. Indeed, we have for every $x \in D$,
\begin{equation}  \nonumber
\begin{split}
\sum_{j=1}^\infty  \rho_j\, \max_{1 \le i \le m}|\partial_{x_i}\psi_j(x)| \ + \ |\overline{a}|_{\Wi}
 \ &\le \ 
\lambda \sum_{j \in E} |\psi_j|_{\Wi} \ + \ e^q \sum_{j \in F} |\psi_j|_{\Wi} \\
 \ &+ \ 
\frac{r}{4}\sum_{j \in F} \frac{s_j}{|s_F|\|\psi_j\|_{\Wi}}|\psi_j|_{\Wi} \ + \ |\overline{a}|_{\Wi} \\
 \ &\le \ 
\sum_{j \in E} \|\psi_j\|_{\Wi} \ + \ \frac{r}{6} \ + \ \frac{r}{12} \ + \ \frac{r}{4}\ + \ \|\overline{a}\|_{\Wi}
\ = \ B.
\end{split}
\end{equation}
By applying Lemma \ref{lemma[t_s]-W} we have that
\begin{equation}  \nonumber
\|t_s\|_W
\ \le \ 
C_{r/2,B} \, \rho^{-s}, \quad s \in \FF.
\end{equation}
Hence, we derive that
\begin{equation}  \label{ineq[t_s]2}
p_s\|t_s\|_W
\ \le \ C_{j_0,\lambda,q} \, C_{r/2,B} \, \sigma_s^{-1}, \quad s \in \FF,
\end{equation}
where
\begin{equation}  \label{sigma[taylor]2}
\sigma_s
:= \
\left(\prod_{j \in E}\left(\frac{2\lambda}{\lambda+1}\right)^{s_j}\right)
\left(\prod_{j \in F}\rho_j^{s_j}(s_j + 1)^q\right).
\end{equation}
In a way similar to (4.23)--(4.26) in the proof of \cite[Theorem 4.3]{CCS13} we can prove the estimate 
\begin{equation}  \nonumber
 \sigma_s^{-1}
\ \le \ 
 \tilde{\sigma}_s^{-1}, \quad s \in \FF,
\end{equation}
where
\begin{equation}  \nonumber
\tilde{\sigma}_s
:= \
\left(\prod_{j \in E}\left(\frac{2\lambda}{\lambda+1}\right)^{s_j}\right)
\left(\prod_{j \in F}\left(\frac{|s_F|d_j}{s_j}\right)^{-s_j}\right),
\quad s \in \FF,
\end{equation}
and
\begin{equation}  \nonumber
d_j
:= \
\frac{4e^q}{r}\|\psi_j\|_{\Wi}.
\end{equation}
By the construction we have also that
\begin{equation}  \nonumber
\sum_{j \in F} d_j
\ \le \ 
\frac{1}{3}.
\end{equation}
Hence, by \cite[Lemma 7.1]{CDS10} we can conclude that 
\begin{equation}  \label{ell_p-summability}
\big(\|\psi_j\|_{\Wi}\big)_{j \in \NN} \in \ell_p(\NN), \ 0 < p \le 1, \
\Longrightarrow \
\big(\tilde{\sigma}_s^{-1}\big)_{s \in \FF}, \,\big(\sigma_s^{-1}\big)_{s \in \FF} \in \ell_p(\FF).
\end{equation}

\subsection{Linear collective collocation approximation}

For a finite lower subset $G$ in $ \ZZ_+ \times \FF$, we define the linear operator
\[
\Ii_G 
:= \
\sum_{(k,s) \in G} \delta_k \Delta_s.
\]
which is a mapping from $L_\infty(\IIi,V)$ to the subspace $\Vv^{{\rm T}}(G)$.
We want to approximate $u(y)$ by $\Ii_{G(T)}u(y)$ in the norm of $L_\infty(\IIi,V)$. 

\begin{theorem} \label{theorem[I-approximation]}
Let Assumption $\operatorname{(i)}$ hold.  Assume that there exists $0 < p < 1$ such that the sequence 
$\big(\|\psi_j\|_{\Wi}\big)_{j \in \NN}$ belong to $\ell_p(\NN)$.
Let the sequence  $(\xi_j)_{j \in \ZZ_+}$ be chosen so that $\lambda_j \le (j + 1)^{q-1}$ for some 
$q > 1$.
Let $\sigma:= \big(\sigma_s\big)_{s \in \FF}$ be the sequence defined by \eqref{sigma[taylor]2}.
For $T > 0$, consider the set $G(T) = G_{p,\sigma}(T)$ as in \eqref{def[G(T)]}.  
Then we have for every $T > 0$,
\begin{equation}  \nonumber
\Big\|u - \Ii_{G(T)} u\Big\|_{L_\infty(\IIi,V)}
\ \le \ 
C\, T^{- \min(1/p -1,\alpha)},
\end{equation}
where 
\[
C:= \ \big(C_D(2^\alpha + 1) + C_{j_0,\lambda,q}\big)\,
\frac{ C_{r/2,B}}{2^{\alpha^*} - 1} \, \big\|\big(\sigma_s^{-1}\big)\big\|_{\ell_p(\FF)}^p.
\]
\end{theorem}

\begin{proof}
Let $T > 0$ be given. For $k \in \ZZ_+$, put
\begin{equation} \label{[Lambda_k(T)]}
\Lambda_k:= \{s \in \FF: \,  (k,s) \in G(T) \}
\ = \ 
\{s \in \FF: \sigma_s^\beta \le 2^{- k}T\}.
\end{equation}
Observe that $\Lambda_k = \emptyset$ for all $k > k^*:= \lfloor \log_2 T \rfloor$, and consequently, 
we have that
\begin{equation}  \label{Eq[Ii]}
\Ii_{G(T)} u
\ = \ 
\sum_{k=0}^{k^*} \delta_k \Big(\sum_{s \in \Lambda_k} \Delta_s \Big)u
\ = \ 
\sum_{k=0}^{k^*} \delta_k I_{\Lambda_k}u.
\end{equation}
Moreover, by the construction $\big(\sigma_s\big)_{s \in \FF}$ is an increasing sequence and, consequently,   $\Lambda_k$ are lower sets. This yields that the sequence $\big\{\Lambda_k\big\}_{k=0}^{k^*}$ is nested in the inverse order, i.e., $\Lambda_{k'} \subset \Lambda_k$ if $k' > k$, and 
$\Lambda_0$ is the largest and $\Lambda_{k^*} = \{0_\FF\}$.
Observe that $I_{\Lambda_k} y^s = y^s$ for every $s \in \Lambda_k$ and
$\Delta_s y^{s'} = 0$ for every $s \not\le s'$.  By \eqref{ell_p-summability} and Lemma \ref{lemma[t-convergence]} the Taylor series unconditionally converges. Hence, we can write
\begin{equation}\nonumber
I_{\Lambda_k}u(y)
\ = \
I_{\Lambda_k}\Big(\sum_{s \in \FF} t_s \,y^s \Big)
\ =  \
\sum_{s \in \FF} t_s \,I_{\Lambda_k} y^s
\ =  \
\sum_{s \in \Lambda_k}  t_s \, y^s
\ + \ \sum_{s \not\in \Lambda_k} t_s \, I_{\Lambda_k \cap R_s}\, y^s.
\end{equation}
Therefore, from  \eqref{Eq[Ii]} we derive that
\begin{equation}\nonumber
\begin{split}
\Ii_{G(T)} u(y)
\ &= \
\sum_{k=0}^{k^*}  \sum_{s \in \Lambda_k} \delta_k(t_s) \, y^s
\ + \ 
\sum_{k=0}^{k^*}  \sum_{s \not\in \Lambda_k} \delta_k(t_s) \, I_{\Lambda_k \cap R_s}\, y^s
\\[1.5ex]
\ &= \
\Ss^{{\rm T}}_{G(T)} u(y)
\ + \ 
\sum_{(k,s) \not\in G(T)} \delta_k(t_s) \, I_{\Lambda_k \cap R_s}\, y^s.
\end{split}
\end{equation}
This together with \eqref{[T-series]} implies that
\begin{equation}\nonumber
u(y) \ - \ \Ii_{G(T)} u(y)
\ = \
u(y) \ - \ \Ss^{{\rm T}}_{G(T)} u(y)
\ - \ 
\sum_{(k,s) \not\in G(T)} \delta_k(t_s) \, 
 I_{\Lambda_k \cap R_s}\, y^s.
\end{equation}
Hence, we have 
\begin{equation} \label{[|u-Iu|<]1}
\big\|u - \Ii_{G(T)} u\big\|_{L_\infty(\IIi,V)}
\ \le \
\big\|u - \Ss^{{\rm T}}_{G(T)} u\big\|_{L_\infty(\IIi,V)} 
 +  
\sum_{(k,s) \not\in G(T)} \|\delta_k(t_s)\|_V \, 
\big\|I_{\Lambda_k \cap R_s}\, y^s\big\|_{L_\infty(\IIi)}.
\end{equation}
From \eqref{ineq[t_s]2} and \eqref{ell_p-summability} it follows that there holds Assumption~(ii) for the number $p$,  the sequence $\sigma:= \big(\sigma_s\big)_{s \in \FF}$ defined in \eqref{sigma[taylor]2} and 
$M = C_{j_0,\lambda,q} \, C_{r/2,B}$. Therefore, 
by Theorem  \ref{theorem[T-approximation](1)} we obtain for every $T > 0$,
\begin{equation}  \label{ineq[u(y) - Ss_{G(T)} u(y)]}
\Big\|u(y) - \Ss^{{\rm T}}_{G(T)} u(y)\Big\|_{L_\infty(\IIi,V)}
\ \le \ 
C'\, T^{- \min(1/p -1,\alpha)},
\end{equation}
where
\begin{equation}  \nonumber
C' := \  
C_{j_0,\lambda,q} \, C_{r/2,B}\, C_D \,
\frac{2^\alpha + 1}{2^{\alpha^*} - 1} \, \big\|\big(\sigma_s^{-1}\big)\big\|_{\ell_p(\FF)}^p.
\end{equation}
For the second sum in \eqref{[|u-Iu|<]1} we have the estimate
\begin{equation} \label{[|u-Iu|<]}
 \sum_{(k,s) \not\in G(T)} \|\delta_k(t_s)\|_V \, \big\|I_{\Lambda_k \cap R_s}\, y^s\big\|_{L_\infty(\IIi)}
\ \le \
 \sum_{(k,s) \not\in G(T)} 2^{-\alpha k} \|t_s\|_W \, 
\Ll_{\Lambda_k \cap R_s}.
\end{equation}
Lemma \ref{lemma[L-constant]} yields that for every $s \in \FF$,
\begin{equation}\nonumber
\Ll_{\Lambda_k \cap R_s}
\ \le \
|\Lambda_k \cap R_s|^q
\ \le \
|R_s|^q
\ = \
\prod_{j \in \NN} (1 + s_j)^q 
\ = \ p_s
\end{equation}
which together with \eqref{ineq[t_s]2} and  \eqref{[|u-Iu|<]} gives 
\begin{equation} \label{[sum_{(k,s) not-in G(T)}]}
\begin{split}
 \sum_{(k,s) \not\in G(T)}  \|\delta_k(t_s)\|_V \, \big\|I_{\Lambda_k \cap R_s}\, y^s\big\|_{L_\infty(\IIi)}
\ &\le \
 \sum_{(k,s) \not\in G(T)} 2^{-\alpha k} p_s\|t_s\|_W
\\[1ex]  
\ &\le \
C_{j_0,\lambda,q}\,C_{r/2,B} \,\sum_{(k,s) \not\in G(T)} 
2^{-\alpha k} \, \sigma_s^{-1}.
\end{split}
\end{equation}
From Assumption (ii) we know that $\big(\sigma_s^{-1}\big)_{s \in \FF}$ belongs to $\ell_1(\FF)$. Hence, by applying Lemma  \ref{lemma[sum-notin]}  
to the sum in the right-hand side of \eqref{[sum_{(k,s) not-in G(T)}]} we obtain
\begin{equation}\nonumber
 \sum_{(k,s) \not\in G(T)}  \|\delta_k(t_s)\|_V \, \big\|I_{\Lambda_k \cap R_s}\, y^s\big\|_{L_\infty(\IIi)}
\ \le \
C_{j_0,\lambda,q}\,C_{r/2,B} \, 
\frac{1}{2^{\alpha^*} - 1} \, \big\|\big(\sigma_s^{-1}\big)\big\|_{\ell_p(\FF)}^p\, T^{- \min(1/p -1,\alpha)}.
\end{equation}
Combining the last estimate, \eqref{[|u-Iu|<]1} and \eqref{ineq[u(y) - Ss_{G(T)} u(y)]} proves the theorem.
\end{proof}

Similarly to the proof of Theorem \ref{theorem[Galerkin-approximation]},
from Theorem \ref{theorem[I-approximation]} and Lemma \ref{lemma[sumG(T)<]} we derive the following
 
\begin{theorem} 
Let  the assumptions and notation of Theorem \ref{theorem[I-approximation]} hold.
For any $n \in \NN$, let $T_n$ be the number defined as in \eqref{[T_n]}
and put $\Vv^{{\rm T}}_n:= \Vv^{{\rm T}}\big(G(T_n)\big)$, \  $\Ii_n:= \Ii_{G(T_n)}$. Then there holds the following.
\begin{itemize}
\item 
$\big\{\Vv^{{\rm T}}_n\big\}_{n \in \ZZ_+}$ is a nested sequence of subspaces in 
$L_\infty(\IIi,V)$ and $\dim \Vv^{{\rm T}}_n \le n$;

\item
$\big\{\Ii_n\big\}_{n \in \ZZ_+}$ is a sequence of linear bounded operators
from $L_\infty(\IIi,V)$ into $\Vv^{{\rm T}}_n$;  and

\item
For every $n \in \NN$, 
\begin{equation}  \nonumber
\big\|u - \Ii_n u\big\|_{L_\infty(\IIi,V)}
\ \le \ 
C\, n^{- \min(1/p -1,\alpha)},
\end{equation}
with the same $\alpha$ as in the convergence rate  of the approximation 
in Assumption $\operatorname{(i)}$,
where
\begin{equation}  \nonumber
C:= \ 
4^\alpha \,\big(C_D(2^\alpha + 1) + C_{j_0,\lambda,q}\big)\,
\frac{ C_{r/2,B}}{2^{\alpha^*} - 1} \,  \big\|\big(\sigma_s^{-1}\big)\big\|_{\ell_p(\FF)}^{p\alpha}.
\end{equation}
Moreover, if in addition,  $p = \frac{1}{1+\alpha}$, then we have that
\begin{equation}   \label{[Pp_n-approximation](1)}
\big\|u - \Ii_n u\big\|_{L_\infty(\IIi,V)}
\ \le \ 
C\, n^{- \alpha},
\end{equation}
\end{itemize}
\end{theorem}

Let us show that the collective polynomial interpolation method 
$\Ii_{G(T)}$ is a collocation method and how to construct it. From the proof of 
Theorem~\ref{theorem[I-approximation]} we know that
for the sets $\Lambda_k$ introduced in \eqref{[Lambda_k(T)]},
$\Lambda_k = \emptyset$ for all $k > k^*:= \lfloor \log_2 T\rfloor$, and therefore,
\begin{equation}  \nonumber
\Ii_{G(T)} u
\ = \ 
\sum_{k=0}^{k^*} \delta_k \Big(\sum_{s \in \Lambda_k} \Delta_s \Big)u
\ = \ 
\sum_{k=0}^{k^*} \delta_k I_{\Lambda_k}u.
\end{equation}
Moreover, $\Lambda_k$ are lower sets nested in the inverse order, i.e.,  
$\Lambda_0 \supset \Lambda_1 \cdots \supset \Lambda_{k^*}$ and 
\begin{equation} \nonumber
\Lambda_0
\ = \ 
\{s \in \FF: \sigma_s^p \le T\}, \quad \Lambda_{k^*} = \{0_{\FF}\}.
\end{equation}
As mentioned above, $\Lambda_k$ can be seen as the section $\{s^0,...,s^{j_k}\}$ of a sequence  $(s^j)_{j \in \NN}$. 
Consequently,
 \begin{equation}  \nonumber
I_{\Lambda_k}u(y)
:= \ 
\sum_{s \in \Lambda_k} u_s h_s(y)
\ = \
\sum_{j=0}^{j_k} u_{s^j} h_{s^j}(y),
\end{equation}
where $u_{s^j}$ are  recursively constructed by the algorithm
\begin{equation} \label{u_s}
u_{s^0} 
:= \
u(\xi_{0_{\FF}}),
\quad 
u_{s^{j+1}}
:= \
u(\xi_{s^{j+1}}) - \sum_{j'=1}^j u_{s^{j'}}\, h_{s^{j'}}(\xi_{s^{j+1}}).
\end{equation}
Observe that $I_{\Lambda_k}$ can be constructed from $I_{\Lambda_{k+1}}$ starting with 
$ I_{\Lambda_{j^*}} := u(\xi_{0_{\FF}})$. Hence, 
$\Ii_{G(T)}u $ is a collocation method based on the particular solutions
$u(\xi_s)$, $s \in \Lambda_0$, of the forms
\begin{equation}   \label{I_G(T)u}
\Ii_{G(T)}u 
\ = \
\sum_{k=0}^{k^*} \delta_k I_{\Lambda_k}u 
\ = \ 
\sum_{k=0}^{k^*} \sum_{j=0}^{j_k}  \delta_k(u_{s^j})\, h_{s^j}.
\end{equation}

Finally, we give an analysis on the computational cost of the approximation $\Ii_{G(T)}u$ to the solution $u$. 
If we take any point $y \in \II^\infty$ and use the operator $P_l$ in Assumption~$\operatorname{(i)}$ to approximate the particular solution $u(y)$, then $l$ can be considered as the computational cost of this approximation. This yields that the computational cost $N$ of the operator $\delta_k(u(y))$ does not exceed $2^k$. Hence, the computational cost of the term 
$\delta_k I_{\Lambda_k}u$ does not exceed $2^k|\Lambda_k|$, and consequently by the formulas 
\eqref{u_s}--\eqref{I_G(T)u} and Lemma 
\ref{lemma[sumG(T)<]} the computational cost $N$ of the approximation $\Ii_{G(T)}u$ does not exceed
\begin{equation}   \nonumber
N 
\ \le \ 
\sum_{k=0}^{k^*} 2^k |\Lambda_k| 
\ = \ 
\sum_{(k,s) \in G(T)} 2^k
\ \le \
C \, T,
\end{equation}
where $C:= 2\big\|\big(\sigma_s^{-1}\big)\big\|_{\ell_p(\FF)}^p$. Thus, if the assumption of 
Theorem~\ref{theorem[I-approximation]} holds for $p = \frac{1}{1+\alpha}$, then we can conclude that with the computational cost $N$ we achieve the approximation error
\begin{equation}  \nonumber
\Big\|u - \Ii_{G(T)} u\Big\|_{L_\infty(\IIi,V)}
\ \le \ 
C\, N^{- \alpha}
\end{equation}
with an absolute positive constant $C$.

\section{Galerkin approximation}
\label{Galerkin approximation}

Let us define a 
probability measure $\mu$ on $\IIi$ as the
infinite tensor product measure  of the univariate uniform probability measures on the one-dimensional $\II$:
\begin{equation} \nonumber
\mbox{d} \mu(y) 
\ = \ 
\bigotimes_{j \in \ZZ} \frac{1}{2}dy_j.
\end{equation}
The sigma algebra $\Sigma$ for $\mu$ is generated by the finite rectangles
$ 
\prod_{j \in \NN} I_j
$
where only
a finite number of the $I_j$ are different from $\II$.  Then $(\IIi, \Sigma, \mu)$ is a probability space. 
Let $L_2(\IIi, \mu)$ denote the Hilbert space of functions on $\IIi$ equipped
with the inner product
\begin{equation} \nonumber
\langle f,g \rangle 
:= \
\int_{\IIi} f(y) \overline{g(y)} \, \mbox{d} \mu(y).
\end{equation}

Consider two types of Legendre univariate polynomials expansions different only in their normalization for basis.  The univariate Legendre  basis $(P_n)_{n \in \NN}$ is defined with $L_\infty(\II)$-normalization:
$\|P_n\|_{L_\infty(\II)}  = 1$.
The orthonormal basis $(L_n)_{n \in \ZZ_+}$ in $L_2(\II, dy/2)$ for which 
$L_n = \sqrt{2n + 1}\, P_n$ and $\|L_n\|_{L_2(\II)}  = 1$.
Observe that $L_0 = P_0 = 1$ and there hold the Rodrigues formulas  
\begin{equation} \label{R-F}
P_n(t)
\ = \
\frac{(-1)^n}{2^n n!} \, \frac{d^n}{dt^n}\Big[(1 - t^2)^n\Big].
\end{equation}

Denote by $\FF$ the subset in $\ZZip$ of all $s$ such that $\operatorname{supp}(s)$ is finite, where 
$\operatorname{supp}(s)$ is the support of $s$, that is the set of all $j \in \NN$ such that 
$s_j \not=0$. 
We define the tensor products of these polynomials 
\begin{equation}\nonumber
P_s (y)
:= \
\prod_{j \in \NN} P_{s_j} (y_j) \quad \mbox{and} \quad 
L_s (y)
:= \
\prod_{j \in \NN} L_{s_j} (y_j), \quad s \in \FF.
\end{equation}
Then $(L_s)_{s \in \FF}$ is an orthonormal basis of $L_2(\IIi, \mu)$. 

Let $X$ be a Banach space and $1 \le p \le \infty$. Denote by $L_\infty(\IIi,X)$  the space of all mappings $v$ from 
$\IIi$ to $X$ for which the following norm is finite
\begin{equation} \nonumber
\|v\|_{L_\infty(\IIi,X)}
:= \
\sup_{y \in \IIi} \|v(y)\|_X.
\end{equation}
We also use the notation  
\begin{equation} \nonumber
|v|_{L_\infty(\IIi,X)}
:= \
\sup_{y \in \IIi} |v(y)|_X
\end{equation}
for a semi-norm $|v(y)|_X$ in $X$ if any. he probability measure $\mu$ induces  the Bochner space $L_p(\IIi,X,\mu)$ of $\mu$-measurable mappings $v$ from $\IIi$ to $X$ which are  $p$-summable. The norm in $L_p(\IIi,X,\mu)$ is defined by
\begin{equation} \nonumber
\|v\|_{L_p(\IIi,X,\mu)}
:= \
\left(\int_{\IIi} \|v(\cdot,y)\|_X^p \, \mbox{d} \mu(y) \right)^{1/p},
\end{equation}
with the change to ess sup norm when $p=\infty$. For simplicity we identify 
$L_\infty(\IIi,X,\mu)$ with $L_\infty(\IIi,X)$.
For a Hilbert space $X$ and $p=2$, the Bochner space $L_2(\IIi,X,\mu)$ coincides with the tensor product 
$X \otimes L_2(\IIi, \mu)$. 

Due to \eqref{|u(y)|_V <} there hold the inclusions $u \in L_\infty(\IIi,V) \subset L_2(\IIi,V,\mu)$. Hence it follows that $u$ admits the unique expansion 
\begin{equation} \label{Legendre-series}
u
\ = \
\sum_{s \in \FF} u_s \, P_s
\ = \
\sum_{s \in \FF} v_s \, L_s,
\end{equation} 
converging in the Hilbert space $L_2(\IIi,V,\mu)$, where the Legendre coefficients $u_s, v_s,$ are defined by
\begin{equation} \label{v_s,u_s}
v_s
:= \
\langle u, L_s \rangle, \quad  
u_s
:= \ \prod_{j \in \NN} (2s_j + 1)^{1/2} \, v_s, \quad s \in \FF.
\end{equation} 
Moreover, from the identity 
$L_2(\IIi,V,\mu) = V \otimes L_2(\IIi, \mu)$ it follows Parseval's identity
\begin{equation} \label{ParsevalIdV}
\|u\|_{L_2(\IIi,V,\mu)}^2
\ = \
\sum_{s \in \FF} \|v_s\|_V^2. 
\end{equation} 
Similarly, assume that $a \in L_\infty(\IIi,\Wi)$, then by \eqref{PDE[u_W<]} we have the inclusions 
$u \in L_\infty(\IIi,W) \subset L_2(\IIi,W,\mu)$, and therefore,  the convergence of the Legendre expansion 
\eqref{Legendre-series} in the Hilbert space $L_2(\IIi,W,\mu)$ and Parseval's identity 
\begin{equation} \label{ParsevalIdW}
\|u\|_{L_2(\IIi,W,\mu)}^2
\ = \
\sum_{s \in \FF} \|v_s\|_W^2. 
\end{equation}

For $s \in \FF$ with $\operatorname{supp}(s) \subset \{1,2,...,J\}$, we define the partial derivative
\[
\partial_y^s u
:= \
\frac{\partial^{|s|} u} {\partial^{s_1}{y_1} \cdots \partial^{s_J} {y_J}},
\]
where $|s| := \sum_{j =1}^J |s_j|$. 

It is known \cite{CDS10} that at any $y \in \IIi$, the function $y\mapsto u(y)$ admits a  partial derivative $\partial_y^s u$. Moreover,  starting with $u(y)$ which is the unique solution in $V$ of the variational equation \eqref{var-eq[u(z)]},
we can recursively find all $\partial_y^s u(y)$ as the unique solution of the variational equation
\begin{equation} \label{partial_y(u)}
\int_{D} a(y)(x)\nabla \partial_y^s u(y) (x) \cdot \nabla v(x) \, \mbox{d}x
\ = \
- \sum_{j: \ s_j \not= 0} s_j \int_{D} \psi_j(x) \, \nabla \partial_y^{s-e^j} u(y)(x) \cdot \nabla v(x) \, \mbox{d}x. \quad \forall v \in V.
\end{equation}
By use of \eqref{R-F} we derive from 
\eqref{partial_y(u)} by inductive integration by parts in the variables $y_j$ the formulas for the Legendre coefficients 
\begin{equation} \label{eq[v_s]}
v_s
\ = \
\frac{1}{s!} \, \prod_{j: \ s_j \not=0}\frac{(2s_j + 1)^{1/2}}{2^{s_j}}
\int_{\IIi} \partial_y^s u(y) \prod_{j: \ s_j \not=0} (1 - y_j^2)^{s_j} \mbox{d} \mu(y),
\end{equation}
where  $s ! := \prod_{j =1}^J  s_j !$.

Since $u \in L_2(\IIi,V,\mu)$, it can be defined as the unique solution of the variational problem: 
Find  $u \in L_2(\IIi,V,\mu)$ such that
\begin{equation} \nonumber
B(u,v)
\ = \
F(v) \quad \forall v \in  L_2(\IIi,V,\mu),
\end{equation}
where
\begin{equation} \nonumber
\begin{split}
B(u,v)
&:= \
\int_{\IIi}\int_{D} a(x,y)\nabla u(x,y) \cdot \nabla v(x,y) \, \mbox{d}x\, \mbox{d}\mu(y),
\\[1ex]
F(v)
&:= \
\int_{\IIi}\int_{D} f(x) \, v(x,y) \, \mbox{d}x\, \mbox{d}\mu(y).
\end{split}
\end{equation}

For a subset $G$ in $ \ZZ_+ \times \FF$, 
denote by $\Vv^{{\rm L}}(G)$ the subspace in $L_\infty(\IIi,V)$ of  all functions $v$
of the form
\begin{equation} \nonumber
v(y)(x)
\ = \
\sum_{(k,s) \in G} v_k(x) \, L_s(y), \quad y \in \IIi, \quad v_k \in V_{2^k},
\end{equation}
and define the linear operator $\Ss^{{\rm L}}_G: \, L_\infty(\IIi,V) \to \Vv^{{\rm L}}(G)$ by
\[
\Ss^{{\rm L}}_G u(y)(x)
:= \
\sum_{(k,s) \in G} \delta_k (v_s)(x) \, L_s(y)
\ = \
\sum_{(k,s) \in G} \delta_k (u_s)(x) \, P_s(y).
\]

If $G$ is a finite set, we define the {\em Galerkin approximation} $u_G$ to $u$ as the unique solution to the problem:
Find  $u_G \in \Vv^{{\rm L}}(G)$ such that
\begin{equation} \nonumber
B(u_G,v)
\ = \
F(v) \quad \forall v \in  \Vv^{{\rm L}}(G).
\end{equation}

By C\'ea's lemma we have the estimate 
\begin{equation} \nonumber
\|u - u_G\|_{L_2(\IIi,V,\mu)}
\ \le \
\sqrt{\frac{R}{r}}\,\inf_{v \in \Vv^{{\rm L}}(G)} \|u - v\|_{L_2(\IIi,V,\mu)},
\end{equation}
and consequently,
\begin{equation} \label{|u - u_G|<}
\|u - u_G\|_{L_2(\IIi,V,\mu)}
\ \le \
\sqrt{\frac{R}{r}} \, \|u - \Ss^{{\rm L}}_G u\|_{L_2(\IIi,V,\mu)}.
\end{equation}

For linear collective Galerkin approximations we need the following assumption.

\smallskip
\noindent
{\bf Assumption (iii)}: There exist a sequence $\sigma = (\sigma_s)_{s \in \FF}$ 
and a constant $M$ such that the sequence
$(\sigma_s^{-1})_{s \in \FF}$ belongs to $\ell_{p(\alpha)}(\FF)$ for $p(\alpha) = \frac{2}{1+2\alpha}$ and
\begin{equation}  \nonumber
\|v_s\|_W
\ \le M \, \sigma_s^{-1}, \quad s \in \FF.
\end{equation}

\begin{theorem} \label{theorem[V-approximation]}
Let Assumptions $\operatorname{(i)}$ and $\operatorname{(iii)}$ hold and  $a \in L_\infty(\IIi,\Wi)$. 
For $T > 0$, consider the set $G(T) = G_{p,\sigma}(T)$ as in \eqref{def[G(T)]} for $p = \frac{2}{1+2\alpha}$.
Then we have for every $T > 0$,
\begin{equation}  \nonumber
\|u - u_{G(T)}\|_{L_2(\IIi,V,\mu)}
\ \le \
\sqrt{\frac{R}{r}} \,
\Big\|u - \Ss^{{\rm L}}_{G(T)} u)\Big\|_{L_2(\IIi,V,\mu)}
\ \le \ 
C\,\sqrt{\frac{R}{r}} \, T^{- \alpha},
\end{equation}
where
\begin{equation} \nonumber
C:= \  M\, C_D \,
\frac{2^\alpha + 1}{2^{\alpha} - 1} \, \big\|\big(\sigma_s^{-1}\big)\big\|_{\ell_p(\FF)}^{p/2}.
\end{equation}
\end{theorem}

\begin{proof}
We preliminarily show that
\begin{equation} \label{lim-to-infty}
\lim_{N \to \infty}
\|u -  \Ss^{{\rm L}}_{G_N}(u) \|_{L_2(\IIi,V,\mu)}
\ = \
0,
\end{equation}
where $G_N := \{(k,s) \in \ZZ_+ \times \FF: \, 0 \le k \le N\}$. Obviously, by the definition,  
\begin{equation} \nonumber
\Ss^{{\rm L}}_{G_N}(u)
\ = \  \sum_{s \in \FF} \sum_{k=0}^N \delta_k (v_s) \, L_s
\ = \
\sum_{s \in \FF} P_{2^N}(v_s) \, L_s.
\end{equation} 
By the assumptions we have the inclusion $u \in  L_2(\IIi,W,\mu) \subset L_2(\IIi,V,\mu)$.
From the uniform boundedness of the operators $P_{2^N}$ and \eqref{ParsevalIdV}
\begin{equation} \nonumber
\|\Ss^{{\rm L}}_{G_N}(u)\|_{L_2(\IIi,V,\mu)}^2
\ = \
\sum_{s \in \FF} \|P_{2^N}(v_s)\|_V^2
\ \le \
C_D^2 \sum_{s \in \FF} \|v_s\|_V^2
\ = \
C_D^2\, \|u\|_{L_2(\IIi,V,\mu)}^2.
\end{equation} 
This means that $\Ss^{{\rm L}}_{G_N}(u) \in L_2(\IIi,V,\mu)$. Hence, by \eqref{ParsevalIdV}, Assumption $\operatorname{(i)}$ and 
\eqref{ParsevalIdW} we deduce that
\begin{equation} \nonumber
\|u - \Ss^{{\rm L}}_{G_N}(u)\|_{L_2(\IIi,V,\mu)}^2
\ = \
\sum_{s \in \FF}\|v_s -P_{2^N}(v_s)\|_V^2
\ \le \
C_D^2 \, 2^{-2\alpha N} \sum_{s \in \FF} \|v_s\|_W^2
\ = \
C_D^2 \, 2^{-2\alpha N} \|u\|_{L_2(\IIi,W,\mu)}^2
\end{equation} 
which prove \eqref{lim-to-infty}.

Let $T$ be given and $\varepsilon$ arbitrary positive number. Then since $G(T)$ is finite from the definition of 
$G_N$ and \eqref{lim-to-infty} there exists $N = N(T,\varepsilon)$ such that 
$G(T) \subset G_N$ and 
\begin{equation}\label{ineq1}
\|u - \Ss^{{\rm L}}_{G_N}(u)\|_{L_2(\IIi,V,\mu)} 
\ \le \ 
\varepsilon.
\end{equation}
By the triangle inequality,
\begin{equation} \label{ineq2}
\|u - \Ss^{{\rm L}}_{G(T)} u\|_{L_2(\IIi,V,\mu)}
\ \le \ 
\|u - \Ss^{{\rm L}}_{G_N}(u)\|_{L_2(\IIi,V,\mu)}
\ + \
\|\Ss^{{\rm L}}_{G_N}(u) - \Ss^{{\rm L}}_{G(T)} u\|_{L_2(\IIi,V,\mu)}.
\end{equation} 
We have by \eqref{ParsevalIdV} and \eqref{delta-approx-property} that
\begin{equation} \nonumber
\begin{split}
\|\Ss^{{\rm L}}_{G_N}(u) - \Ss^{{\rm L}}_{G(T)} u\|_{L_2(\IIi,V,\mu)}^2
\ &= \
\Big\|\sum_{s \in \FF}  \sum_{k=0}^N \delta_k (v_s) \, L_s - 
\sum_{s \in \FF}  \sum_{2^k > T \sigma_s^{-p}} \delta_k (v_s) \, L_s\Big\|_{L_2(\IIi,V)}^2 
\\[1.5ex]
\ &= \
\Big\|\sum_{s \in \FF}  \ \sum_{T \sigma_s^{-p}<2^k < N} \delta_k (v_s) \, L_s\Big\|_{L_2(\IIi,V)}^2 
\\[1.5ex]
\ &= \
\sum_{s \in \FF}  \ \Big\|\sum_{T \sigma_s^{-p}<2^k < N} \delta_k (v_s)\Big\|_V^2 
\\[1.5ex]
\ &\le \
\sum_{s \in \FF}  \ \Big(\sum_{T \sigma_s^{-p}<2^k < N} \|\delta_k (v_s)\|_V\Big)^2 
\\[1.5ex]
\ &\le \
\sum_{s \in \FF}  \ \Big(\sum_{T \sigma_s^{-p}<2^k < N} (2^\alpha + 1)C_D \, 2^{-\alpha k}\|v_s\|_W\Big)^2 
\\[1.5ex]
\ &\le \
(2^\alpha + 1)^2C_D^2 \, \sum_{s \in \FF} \|v_s\|_W^2 \ \Big(\sum_{2^k >T \sigma_s^{-p}}  2^{-\alpha k}\Big)^2. 
 \end{split}
\end{equation}
Hence, by Assumption $\operatorname{(iii)}$ and the equation $2(1-p\alpha)=p$ we derive that
\begin{equation} \nonumber
\begin{split}
\|\Ss^{{\rm L}}_{G_N}(u) - \Ss^{{\rm L}}_{G(T)} u\|_{L_2(\IIi,V,\mu)}^2
\ &\le \
(2^\alpha + 1)^2C_D^2 \, \sum_{s \in \FF} \sigma_s^{-2} \ \Big(\sum_{2^k >T \sigma_s^{-p}}  2^{-\alpha k}\Big)^2 
\\[1.5ex]
\ &\le \
 \  T^{- 2\alpha}\, M^2\, C_D^2 \,
\frac{(2^\alpha + 1)^2}{(2^{\alpha} - 1)^2} \, \, \sum_{s \in \FF} \sigma_s^{-2(1-p\alpha)}
\\[1.5ex]
\ &= \
 \  T^{- 2\alpha}\, M^2\, C_D^2 \,
\frac{(2^\alpha + 1)^2}{(2^{\alpha} - 1)^2} \, \, \sum_{s \in \FF} \sigma_s^{-p}
\\[1.5ex]
\ &= \
 \ C^2\, T^{- 2\alpha}.
 \end{split}
\end{equation}
which in combining with \eqref{ineq1} and \eqref{ineq2} gives
\begin{equation} \nonumber
\|u - \Ss^{{\rm L}}_{G(T)} u\|_{L_2(\IIi,V,\mu)}
\ \le \ 
\varepsilon
\ + \
C\, T^{- \alpha}
\end{equation}
for arbitrary positive number $\varepsilon$. Hence,
\begin{equation} \nonumber
\|u - \Ss^{{\rm L}}_{G(T)} u\|_{L_2(\IIi,V,\mu)}
\ \le \ 
C\, T^{- \alpha}
\end{equation}
which together with \eqref{|u - u_G|<} proves the theorem. 
\hfill
\end{proof}

We show that under the assumptions of  Theorem \ref{theorem[V-approximation]}, for a given $n \in \NN$, the respective operator $\Ss^{{\rm L}}_{G(T_n)}$ with properly chosen $T_n$ is a bounded linear operator in $L_\infty(\IIi,V)$ of rank $\le n$ which gives the convergence rate of the approximation to $u(y)$ as $n^{- \alpha}$. 

\begin{theorem} \label{theorem[V-Galerkin-approximation]}
Let  the assumptions and notation of Theorem \ref{theorem[V-approximation]} hold.
For any $n \in \NN$, let $T_n$ be the number defined as in \eqref{[T_n]}
and put $\Vv^{{\rm L}}_n:= \Vv^{{\rm L}}\big(G(T_n)\big)$, \  $\Pp_n:= \Ss^{{\rm L}}_{G(T_n)}$, \  $u_n:= u_{G(T_n)}$. Then
\begin{itemize}
\item 
$\big\{\Vv^{{\rm L}}_n\big\}_{n \in \ZZ_+}$ is a nested sequence of subspaces in 
$L_2(\IIi,V,\mu)$ and $\dim \Vv^{{\rm L}}_n \le n$;

\item
$\big\{\Pp_n\big\}_{n \in \ZZ_+}$ is a sequence of linear bounded operators
from $L_2(\IIi,V,\mu)$ into $\Vv^{{\rm L}}_n$;  and

\item
for every $n \in \NN$, 
\begin{equation}  \nonumber
\|u - u_n\|_{L_2(\IIi,V,\mu)}
\ \le \
\sqrt{\frac{R}{r}} \,
\big\|u - \Pp_n u\big\|_{L_2(\IIi,V,\mu)}
\ \le \ 
C\, \sqrt{\frac{R}{r}} \, n^{- \alpha},
\end{equation}
with the same $\alpha$ as in the convergence rate  of the approximation 
in Assumption $\operatorname{(i)}$,
where 
\begin{equation} \nonumber 
C:= \ M\, C_D \,
4^\alpha \ \frac{2^\alpha + 1}{\sqrt{2^{\alpha} - 1}} \, \big\|\big(\sigma_s^{-1}\big)\big\|_{\ell_p(\FF)}^{p\alpha/2}, 
\end{equation}
\end{itemize}
\end{theorem}

\begin{proof}
We have that
\begin{equation} \nonumber
\begin{split}
\dim \Vv^{{\rm L}}(G(T))
\ &\le \ \sum_{(k,s) \in G(T)} \dim V_{2^k} 
\ \le \ \sum_{(k,s) \in G(T)} 2^k
\\[1.5ex]
\ &\le  \
\sum_{\sigma_s^p \ \le \ T} \quad 
\sum_{2^k \ \le T\sigma_s^{-p}} 2^k 
\  \le \
2 \sum_{\sigma_s^p \ \le \ T}  T\sigma_s^{-p} 
\\[1.5ex]
\ &\le  \
2\,T \sum_{s \in  \FF} \sigma_s^{-p}
\ \le \
2 \big\|\big(\sigma_s^{-1}\big)\big\|_{\ell_p(\FF)}^p \,T.
\end{split}
\end{equation} 
Hence, by  \eqref{[T_n]} we derive that
\begin{equation} \label{[dim-P]}
\dim \Vv^{{\rm L}}(G(T_n))
\ \le \ 2  \big\|\big(\sigma_s^{-1}\big)\big\|_{\ell_p(\FF)}^p \, T_n
\ \le \ n.
\end{equation}
On the other hand, by \eqref{[T_n]},
\begin{equation} \nonumber
T_n^{-\alpha} 
\ \le \ 4^\alpha  \big\|\big(\sigma_s^{-1}\big)\big\|_{\ell_p(\FF)}^{p\alpha} \, n^{-\alpha}
\end{equation}
which together with Theorem  \ref{theorem[V-approximation]} 
and \eqref{[dim-P]} completes the proof of the theorem.
\hfill
\end{proof}

\begin{lemma} \label{|v_s|_W<}
Assume that $a \in L_\infty(\IIi,\Wi)$.  Let the constant $K$ be as in \eqref{def[K]} and 
 the sequence $b$ as in \eqref{def[b]}. Define
 the sequence 
\begin{equation} \label{def[d]}
d = (d_j)_{j \in \NN}, \quad d_j := b_j/\sqrt{3}.
\end{equation}
 Then we have
\begin{equation}  \nonumber
\|v_s\|_W 
\ \le \
K \frac{|s|!}{s!}\, d^s, \quad s \in \FF.
\end{equation}
\end{lemma}

\begin{proof}
From \eqref{eq[v_s]} we derive that
\begin{equation} \nonumber
\|v_s\|_W
\ \le \
\frac{3^{-|s|/2}}{s!} \, \|\partial_y^s u\|_{L_\infty(\IIi,W)}
\end{equation}
which combining  with Lemma  \ref{lemma|partial^s-u|_W} prove the lemma.
\hfill
\end{proof}

\begin{corollary}  
Let the function 
$a$  belong to $L_\infty(\IIi,\Wi)$,  $p(\alpha) = \frac{2}{1+2\alpha}$ and the sequence 
$d = (d_j)_{j \in \NN}$ defined in \eqref{def[d]} 
satisfy the condition
\begin{equation}\nonumber
\begin{cases} 
\|d\|_{\ell_1(\NN)} \ < \ 1, \ d \in \ell_{p(\alpha)}(\NN), \ & \ \text{for}\  \alpha \ge 1/2; 
\\[1ex]
\|d\|_{\ell_1(\NN)} \ \le \ 1,  \ & \ \text{for} \ \alpha < 1/2.
\end{cases}
\end{equation}
Then there holds Assumption~$\operatorname{(iii)}$ for $M=K$ and the sequence
\begin{equation}  \nonumber
\sigma:= \big(\sigma_s\big)_{s \in \FF}, \quad
\sigma_s^{-1}
:= \
\frac{|s|!}{s!} d^s.
\end{equation}
\end{corollary}

\begin{proof}
By definition we have that $0 < p(\alpha) \le 1$ for $\alpha \ge 1/2$, and $1 < p(\alpha) < \infty $ for $\alpha < 1/2$.
Hence, by Lemma \ref{ell_p-sumability} 
\begin{equation}\nonumber
\left(\frac{|s|!}{s!}\, d^s\right) \in \ell_{p(\alpha)}(\FF) \ \Longleftrightarrow \
\begin{cases} 
\|d\|_{\ell_1(\NN)} \ < \ 1, \ d \in \ell_{p(\alpha)}(\NN), \ & \ \text{for}\   \alpha \ge 1/2; 
\\[1ex]
\|d\|_{\ell_1(\NN)} \ \le \ 1,  \ & \ \text{for} \ \alpha < 1/2
\end{cases}
\end{equation}
which together with Lemma \ref{|v_s|_W<} proves the corollary.
\hfill
\end{proof}

Notice that according to Assumption~(i) $0 < \alpha \le 1/m$, where $m$ is the dimension of the spatial domain $D$.
Hence the inequality $\alpha \ge 1/2$ may hold only for $m=1,2$, and $\alpha < 1/2$ for all $m > 2$. 
This means that in Assumption~(i) the inequality $p(\alpha) \le 1$ may hold only in the case when $m=1,2$, and except this case we always have $p(\alpha) > 1$. 


\section{Legendre approximation}
\label{Legendre approximation}

The collective Legendre approximation is constructed on the basis of  a  representation of the solution $u$ by a series converging unconditionally in $L_\infty(\IIi,V)$ as in the following lemma. 


\begin{lemma} \label{theorem[U-presentation]}
Let Assumption $\operatorname{(i)}$ hold and let the sequence 
$\big(\|\psi_j\|_{\Wi}\big)_{j \in \NN}$ belong to $\in \ell_1(\NN)$. 
Then $(\|u_s\|_W )_{s \in \FF}$ belongs to $\ell_1(\FF)$ and $u(y)$ can be represented as the series 
\begin{equation} \label{[U-series]}
u(y)
\ = \
\sum_{(k,s) \in \ZZ_+ \times \FF} \delta_k (u_s) \, P_s(y), \quad y \in \IIi,
\end{equation}
converging unconditionally in $L_\infty(\IIi,V)$. 
\end{lemma}

\begin{proof}
This theorem can be proven in a similar way to the proof of Lemma \ref{theorem[T-presentation]}. 
\hfill
\end{proof}

For the linear collective Legendre approximation of the solution $u$ we need the following assumption.


\smallskip
\noindent
{\bf Assumption (iv)}: There exist $0 < p < 1$, a sequence $\sigma = (\sigma_s)_{s \in \FF}$ 
and a constant $M$ such that the sequence
$(\sigma_s^{-1})_{s \in \FF}$ belongs to $\ell_p(\FF)$ and
\begin{equation}  \nonumber
\|u_s\|_W
\ \le M \, \sigma_s^{-1}, \quad s \in \FF.
\end{equation}

 The following two theorems can be proven in a similar way to the proofs of Theorems 
\ref{theorem[T-approximation](1)} and \ref{theorem[Galerkin-approximation]}, respectively.
\begin{theorem} \label{theorem[U-approximation](1)}
Let Assumptions $\operatorname{(i)}$ and $\operatorname{(iv)}$ hold. 
For $T > 0$, consider the set $G(T) = G_{p,\sigma}(T)$ as in \eqref{def[G(T)]}.   
Then we have for every $T > 0$,
\begin{equation}  \nonumber
\Big\|u - \Ss^{{\rm L}}_{G(T)} u\Big\|_{L_\infty(\IIi,V)}
\ \le \ 
C\, T^{-\min (1/p - 1, \alpha)},
\end{equation}
where
\begin{equation}  \nonumber
C:= \  M\, C_D \,
\frac{2^\alpha + 1}{2^{\alpha^*} - 1} \, \big\|\big(\sigma_s^{-1}\big)\big\|_{\ell_p(\FF)}^p,
\end{equation}
$\alpha^*:= \alpha$ for $\alpha \le 1/p - 1$, 
and $\alpha^*:= \alpha - 1/p + 1$ for $\alpha > 1/p - 1$.
\end{theorem}

\begin{theorem} \label{theorem[U-Galerkin-approximation]}
Let  the assumptions and notation of Theorem \ref{theorem[U-approximation](1)} hold.
For any $n \in \NN$, let $T_n$ be the number defined as in \eqref{[T_n]}
and put $\Vv^{{\rm L}}_n:= \Vv^{{\rm L}}\big(G(T_n)\big)$, \  $\Pp_n:= \Ss^{{\rm L}}_{G(T_n)}$. Then
\begin{itemize}
\item 
$\big\{\Vv^{{\rm L}}_n\big\}_{n \in \ZZ_+}$ is a nested sequence of subspaces in 
$L_\infty(\IIi,V)$ and $\dim \Vv^{{\rm L}}_n \le n$;

\item
$\big\{\Pp_n\big\}_{n \in \ZZ_+}$ is a sequence of linear bounded operators
from $L_\infty(\IIi,V)$ into $\Vv^{{\rm L}}_n$;  and

\item
for every $n \in \NN$, 
\begin{equation}  \nonumber
\big\|u - \Pp_n u\big\|_{L_\infty(\IIi,V)}
\ \le \ 
C\, n^{-\min (1/p - 1, \alpha)},
\end{equation}
with the same $\alpha$ as in the convergence rate  of the approximation 
in Assumption $\operatorname{(i)}$ and $p$ as in  Assumption $\operatorname{(iv)}$,
where 
\begin{equation} \nonumber
C:= \ M \, C_D \,
4^\alpha \ \frac{2^\alpha + 1}{2^{\alpha^*} - 1} \, \big\|\big(\sigma_s^{-1}\big)\big\|_{\ell_p(\FF)}^{p\alpha}.
\end{equation}
Moreover, if in addition,  $p = \frac{1}{1+\alpha}$ in  Assumption $\operatorname{(iv)}$, then we have that
\begin{equation}   \nonumber
\big\|u - \Pp_n u\big\|_{L_\infty(\IIi,V)}
\ \le \ 
C\, n^{- \alpha}.
\end{equation}
\end{itemize}
\end{theorem}

From Theorems \ref{theorem[U-approximation](1)} and \ref{theorem[U-Galerkin-approximation]} we see that the problem of construction of a linear collective Legendre approximation is reduced to the construction of a sequence
$\sigma = (\sigma_s)_{s \in \FF}$ satisfying Assumption~(iv).

\begin{corollary}
 Let the constant $K$ be as in \eqref{def[K]} and 
 the sequence $b$ as in \eqref{def[b]}. 
Assume that the function  $a \in L_\infty(\IIi,\Wi)$,
there exists $0 < p <1$ such that the sequence  
$\big(\|\psi_j\|_{\Wi}\big)_{j \in \NN}$  belongs to $\ell_p(\NN)$ and 
$\|b\|_{\ell_1(\NN)}  <  1$.
Then there holds Assumption~$\operatorname{(iv)}$ for $p$, $M=K$ and the sequence
\begin{equation}  \nonumber
\sigma:= \big(\sigma_s\big)_{s \in \FF}, \quad
\sigma_s^{-1}
:= \
\frac{|s|!}{s!} b^s.
\end{equation}
\end{corollary}

\begin{proof}
By using of  \eqref{v_s,u_s}, \eqref{eq[v_s]} and Lemma  \ref{lemma|partial^s-u|_W}
we derive that
\begin{equation} \nonumber
\|u_s\|_W 
\ \le \
\frac{1}{s!} \, \|\partial_y^s u\|_{L_\infty(\IIi,W)}
\ \le \
K \frac{|s|!}{s!} b^s
\ = \ 
K \, \sigma_s^{-1}.
\end{equation}
On the other hand, from the assumptions we have that $b \in \ell_p(\NN)$ and 
$\|b\|_{\ell_1(\NN)}  <  1$.
Hence by Lemma \ref{ell_p-sumability} the sequence $(\sigma_s^{-1})_{s \in \FF}$ belongs to $\ell_p(\FF)$.
This proves the corollary.
\hfill
\end{proof}

\section{Concluding remarks}
\label{Concluding remarks}

 \begin{itemize}
\item
We have constructed linear collective methods for Taylor, collocation, Galerkin and Legendre approximations for 
parametric elliptic PDEs \eqref{SPDE} with affine parametric dependence of  diffusion coefficients on the basic of 
a sequence of approximations to one nonparametric elliptic PDEs with a certain error convergence rate. 

\item
These methods are "optimal" in the sense that they give the same error convergence rate of  the inducing approximations for nonparametric elliptic PDEs. 

\item
All the  conditions on the parametric part disappear in the convergence rate and only influence the constant which can be explicitly estimated.

\item
 In constructing these methods, the spatial variables and the parametric variables are not split, but treated collectively.
 
\item
The curse of dimensionality is broken by linear methods.

\item
In the present paper, the parameter $\alpha$ defining the convergence rate of the approximation error in Assumption (i) is restricted by the condition $0 < \alpha \le 1/m$ caused by the restriction of the regularity of the diffusion coefficients $a(y)$, the function $f$ and the domain $D$. However, we can extend our results to the case where $\alpha$ may be arbitrarily large if we require a proper regularity of  $a(y)$, $f$ and $D$. 
 
\item
Hopefully, the approach and methods which have been considered in this paper can be extended to more general problems. 
In a forthcoming paper, we extend them to the parametric elliptic PDEs \eqref{SPDE} with the diffusions coefficients $a(y)$ not necessarily affinely dependent with respect to $y$, as well to a semi-linear extension and to parametric and stochastic parabolic PDEs.
\end{itemize} 

\bigskip
\noindent
{\bf Acknowledgments.}  This work is funded by Vietnam National Foundation for Science and Technology Development (NAFOSTED) under  Grant No. 102.01-2017.05.  A part of this work was done when the author was working as a research professor at the Vietnam Institute for Advanced Study in Mathematics (VIASM). The author  would like to thank  the VIASM  for providing a fruitful research environment and working condition.  He expresses a special thank to Christoph Schwab for valuable remarks and suggestions.


\begin{thebibliography}{99}

\bibitem{ABS09}
R. Andreev, M. Bieri, Ch. Schwab, Sparse tensor discretization of elliptic sPDEs, SIAM J. Sci. Comput.
31, 4281--4304 (2009).

\bibitem{BNT07}
I. Babuška, F. Nobile, R. Tempone, A stochastic collocation method for elliptic partial differential
equations with random input data, SIAM J. Numer. Anal. 45, 1005--1034 (2007).

\bibitem{BTZ04}
I. Babuska, R. Tempone, G.E. Zouraris, Galerkin finite element approximations of stochastic elliptic
partial differential equations, SIAM J. Numer. Anal. 42, 800--825 (2004).

\bibitem{BCM15}
M. Bachmayr, A. Cohen, and G. Migliorati, Sparse polynomial approximation of
parametric elliptic PDEs. Part I: affine coefficients (2015),  arXiv:1509.07045v1 [math.NA].


\bibitem{BCDS17}
M. Bachmayr, A. Cohen, D. D\~ung and C. Schwab,
Fully discrete approximation of parametric and stochastic elliptic PDEs (2017),
 arXiv:1702.03671v1 [math.NA].


\bibitem{BNTT12}
J. B\"ack, F. Nobile, L. Tamellini, R. Tempone, On the optimal polynomial approximation of stochastic
PDEs by Galerkin and Collocation methods, MOX Report 23/2011, Math. Mod. Methods Appl. Sci.
22(9), 1250023 (2012).

\bibitem{Cia78} P. Ciarlet,
{\it The Finite Element Method for Elliptic Problems}, North Holland Publ., 1978.


\bibitem{Chk13}
A. Chkifa, On the Lebesgue constant of Leja sequences for the complex unit disk and of their real
projection, J. Approx. Theory 166, 176--200 (2013).

\bibitem{CCS15}
A. Chkifa, A. Cohen and C. Schwab, Breaking the curse of dimensionality in sparse polynomial
approximation of parametric PDEs, Journal de Mathématiques Pures et Appliquées 103(2),
400-428, 2015.

\bibitem{CCDS13}
A. Chkifa, A. Cohen, R. DeVore, Ch. Schwab, Sparse adaptive Taylor approximation algorithms for
parametric and stochastic elliptic PDEs, Modél. Math. Anal. Numér. 47(1), 253--280 (2013).

\bibitem{CCS13}
A. Chkifa, A. Cohen, and C. Schwab. High-dimensional adaptive sparse polynomial interpolation and
applications to parametric PDEs. Found. Comp. Math 14(2014), 601--633.

\bibitem{CD15}
A. Cohen and R. DeVore, Approximation of high-dimensional parametric PDEs, Acta Numerica
24, 1-159, 2015.

\bibitem{CDS10}
A. Cohen, R. DeVore, Ch. Schwab, Convergence rates of best N-term Galerkin approximations for a
class of elliptic sPDEs, Found. Comput. Math. 10(6), 615--646 (2010).

\bibitem{CDS11}
A. Cohen, R. DeVore, Ch. Schwab, Analytic regularity and polynomial approximation of parametric
and stochastic PDE's, Anal. Appl. 9, 1–37 (2011).


\bibitem{Di15}
D. D\~ung,
Linear collective collocation and Galerkin approximations for parametric and stochastic elliptic PDEs (2015),
 arXiv:1511.03377v5 [math.NA].


\bibitem{DG15}
D. D\~ung and M. Griebel,
Hyperbolic cross approximation in infinite dimensions,
J. Complexity 33(2016), 33-88. 

\bibitem{DGVR17}
D. D\~ung, M.Griebel, N.H. Vu and C. Rieger,
$\varepsilon$-dimension in infinite dimensional hyperbolic cross approximation and 
application to parametric elliptic PDEs (2017),
arXiv:1703.00128v1 [math.NA].


\bibitem{EMPT11}
H. C. Elman, C. W. Miller, E. T. Phipps, and R. S. Tuminaro. Assessment
of Collocation and Galerkin approaches to linear diffusion equations
with random data. International Journal for Uncertainty Quantification,
1(1):19–33, 2011.

\bibitem{FST05}
Ph. Frauenfelder, Ch. Schwab, R.A. Todor, Finite elements for elliptic problems with stochastic coefficients,
Comput. Methods Appl. Mech. Eng. 194, 205--228 (2005).

\bibitem{Gi13}
C.J. Gittelson, An adaptive stochastic Galerkin method for random elliptic operators, Math. Comp.
82, 1515--1541 (2013).

\bibitem{GWZ14}
M. Gunzburger, C. Webster and G. Zang, 
Stochastic finite element methods for partial diferential equations with random input data,
Acta Numerica 23, 521--650 (2014).

\bibitem{HS13a}
M. Hansen, Ch. Schwab, Analytic regularity and nonlinear approximation of a class of parametric
semilinear elliptic PDEs,  Math. Nachr. 286, No. 8–9, 832--860 (2013).

\bibitem{HoS12}
V.H. Hoang and C. Schwab, 
$N$-term Galerkin Wiener chaos approximation rates 
for elliptic PDEs with lognormal Gaussian random inputs, 
Math. Models Methods Appl. Sci. \textbf{24} (2014) 797--826.

\bibitem{MK05}
H. G. Matthies and A. Keese. Galerkin methods for linear and nonlinear
elliptic stochastic partial differential equations. Comput. Methods Appl.
Mech. Engrg., 194(12-16):1295--1331, 2005.

\bibitem{MNST11}
G. Migliorati, F. Nobile, E. von Schwerin, R. Tempone, Analysis of the discrete $L^2$ projection on
polynomial spaces with random evaluations. Report 46/2011, MOX, Politechnico di Milano.

\bibitem{NS13}
V. Nistor, Ch. Schwab, High order Galerkin approximations for parametric, second order elliptic
partial differential equations. Report 2012-22, Seminar for Applied Mathematics, ETH Zürich (to
appear in Math. Models Methods Appl. Sci. 2013).

\bibitem{NTW08a}
F. Nobile, R. Tempone, C.G. Webster, A sparse grid stochastic collocation method for elliptic partial
differential equations with random input data, SIAM J. Numer. Anal. 46, 2309--2345 (2008).

\bibitem{NTW08b}
F. Nobile, R. Tempone, C.G. Webster, An anisotropic sparse grid stochastic collocation method for
elliptic partial differential equations with random input data, SIAM J. Numer. Anal. 46, 2411--2442
(2008).

\bibitem{SG11}
Ch. Schwab and C. Gittelson, Sparse tensor discretizations high-dimensional parametric and stochastic PDEs, 
Acta Numerica 20, 291--467(2011).

\end{thebibliography}
\end{document}